\documentclass[11pt]{amsart}
\usepackage[utf8]{inputenc}
\usepackage[T1]{fontenc}
\usepackage{graphicx}
\usepackage{mathpazo}

\usepackage[usenames, dvipsnames]{color}

\usepackage{amsmath,amssymb,amsthm,enumerate,xspace}
\usepackage[colorlinks=true,citecolor=blue,linkcolor=red]{hyperref}
\usepackage[all]{xy} 
\usepackage{mathrsfs} 
\usepackage{pict2e} 

\title[Kaleidoscopic groups, permutations and dendrites]{Kaleidoscopic groups: permutation groups\\ constructed from dendrite homeomorphisms}
\author[B. Duchesne]{Bruno Duchesne}
\address{Institut Élie Cartan de Lorraine, CNRS \& Université de Lorraine, Nancy, France.}
\email{bruno.duchesne@univ-lorraine.fr}
\thanks{B.D. is supported in part by French projects ANR-14-CE25-0004 GAMME and ANR-16-CE40-0022-01 AGIRA}
\author[N. Monod]{Nicolas Monod}
\address{EPFL, Lausanne, Switzerland.}
\email{nicolas.monod@epfl.ch}
\author[P. Wesolek]{Phillip Wesolek}
\address{Department of Mathematical Sciences, Binghamton University,  Binghamton, USA.}
\email{pwesolek@binghamton.edu}

\date{January 2018} 

\newtheorem{thm}{Theorem}[section]

\newtheorem{prop}[thm]{Proposition}
\newtheorem{lem}[thm]{Lemma}
\newtheorem{cor}[thm]{Corollary}

\theoremstyle{definition}
\newtheorem{defn}[thm]{Definition}

\newtheorem{rmk}[thm]{Remark}

\newtheorem{exam}[thm]{Example}

\newcommand{\Zb}{\mathbf{Z}}
\newcommand{\Nb}{\mathbf{N}}
\newcommand{\Nbe}{\mathbf{N}^{*}}
\newcommand{\Nbee}{\mathbf{N}^{*}_{\geq 3}}
\newcommand{\Rb}{\mathbf{R}}
\newcommand{\Qb}{\mathbf{Q}}

\newcommand{\mc}[1]{\mathcal{#1}}

\newcommand{\Uf}{\mathscr{K}}
\newcommand{\sU}{\mathscr{U}}
\newcommand{\sG}{\mathscr{G}}
\newcommand{\HH}{\mathrm{H}}
\newcommand{\HB}{\mathrm{H}_{\mathrm b}}

\newcommand{\se}{\subseteq}
\newcommand{\sep}{\supseteq}
\newcommand{\acts}{\curvearrowright}

\newcommand{\rest}{|}

\newcommand{\im}{\mathrm{im}}
\DeclareMathOperator{\id}{id}
\newcommand{\rist}{\mathrm{Rist}}

\DeclareMathOperator{\Homeo}{Homeo}

\DeclareMathOperator{\Sym}{Sym}

\DeclareMathOperator{\diam}{diam}
\newcommand{\gdelta}{G$_\delta$\xspace}
\newcommand{\teta}{\vartheta}
\newcommand{\fhi}{\varphi}

\DeclareMathOperator{\Ends}{Ends}
\DeclareMathOperator{\Reg}{Reg}
\DeclareMathOperator{\Br}{Br}

\newcommand{\comp}[1]{\widehat{#1}}

\DeclareMathOperator{\Gr}{Gr}
\newcommand{\inv}{^{-1}}

\newcommand{\grp}[1]{\langle #1 \rangle}
\newcommand{\ol}[1]{\overline{#1}}
\newcommand{\ul}[1]{\underline{#1}}

\newcommand{\wh}[1]{\widehat{#1}}

\begin{document}

\begin{abstract}
Given a transitive permutation group, a fundamental object for studying its higher transitivity properties is the permutation action of its isotropy subgroup. We reverse this relationship and introduce a universal construction of infinite permutation groups that takes as input a given system of imprimitivity for its isotropy subgroup.

This produces vast families of \emph{kaleidoscopic} groups. We investigate their algebraic properties, such as simplicity and oligomorphy; their homological properties, such as acyclicity or contrariwise large Schur multipliers; their topological properties, such as unique polishability.

Our construction is carried out within the framework of homeomorphism groups of topological dendrites.
\end{abstract}

\maketitle

\section{Introduction}
\begin{flushright}
\begin{minipage}[t]{0.7\linewidth}\itshape\small
Ces effets de prisme se modifient \`a l'infini, suivant la place que l'on occupe. Ne vous semble-t-il pas que nous sommes pench\'es sur l'ouverture d'un immense kal\'eidoscope ?
\begin{flushright}\upshape\small
--- Jules Verne, \textit{Le Pays des fourrures.}\\
\end{flushright}
\end{minipage}
\end{flushright}

\bigskip
Consider an arbitrary permutation group $\Gamma<\Sym(n)$ of a finite or countable set of $n$ elements, where $3\leq n \leq \infty$. From this data, we construct canonically a primitive permutation group of an infinite countable set: the \textbf{kaleidoscopic group} $\Uf(\Gamma)$.

This construction is particularly well-suited to explore the worlds that open up between primitivity and double primitivity. For instance, $\Uf(\Gamma)$ is often doubly transitive, but never triply transitive.

The principle behind this exploration is that the functor $\Gamma \mapsto \Uf(\Gamma)$ shifts the attention one level deeper into the action, namely to the action of a point-stabilizer $\Uf(\Gamma)_x$ (upon the complement of that point~$x$). More precisely, the construction is such that $\Uf(\Gamma)_x$ admits a system of imprimitivity isomorphic to the initial permutation group $\Gamma$.

\smallskip
Here are a first few properties of the functor $\Gamma \mapsto \Uf(\Gamma)$ defined in this text.

\begin{thm}\leavevmode\label{thm:general}
\begin{enumerate}[(i)]
\item The abstract group  $\Uf(\Gamma)$ is simple and uniformly perfect.\label{pt:general:simple}
\item The permutation group $\Uf(\Gamma)$ is always primitive; it is doubly transitive if and only if $\Gamma$ is transitive.
\item The permutation group $\Uf(\Gamma)$ is never doubly primitive: its point-stabilizers admit a system of imprimitivity isomorphic to $\Gamma$ and decompose as permutational wreath product over $\Gamma$.\label{pt:general:Gamma}
\end{enumerate}
\end{thm}

\begin{rmk}
In~\eqref{pt:general:simple}, \textbf{uniform perfectness} means that every element is the product of a uniformly bounded number of commutators. We shall show that this number can be taken to be~$3$.
\end{rmk}

In order to appreciate the diversity of kaleidoscopic groups, it is natural to ask how much $\Uf(\Gamma)$ depends upon $\Gamma$. This question can be asked for the permutation groups or for the underlying abstract groups. An intermediate level is to consider the groups endowed with the Polish topology of pointwise convergence.

We obtain the strongest possible answer under a discreteness assumption on $\Gamma$, which is a void assumption unless $n=\infty$.

\begin{thm}\label{thm:intro:isom}
Suppose that the groups $\Gamma<\Sym(n)$  and $\Gamma'<\Sym(n')$ are discrete, which is automatic when $n, n'<\infty$. Then $\Uf(\Gamma)$ and $\Uf(\Gamma')$ are non-isomorphic even as abstract groups, unless $n=n'$ and $\Gamma\cong \Gamma'$ as permutation groups.
\end{thm}

\noindent
In particular, even the trivial group $\Gamma=1$ gives rise to countably many non-isomorphic simple primitive permutation groups $\Uf(1\acts [n])$, where $[n]=\{0, \ldots, n-1\}$.

Another family of examples is the regular permutation group $\Gamma\acts \Gamma$ associated to an arbitrary countable group $\Gamma$; this produces a continuum of non-isomorphic kaleidoscopic groups of very high descriptive complexity (by~\cite{Thomas-Velickovic}).

\smallskip
\begin{center}
* \kern5mm *  \kern5mm *
\end{center}

\smallskip
Our kaleidoscopic groups are constructed as homeomorphisms of \textbf{dendrites}, that is, of locally connected continua containing no simple closed curve. These topological spaces arise naturally in many contexts; for instance, some appear as Julia sets~\cite[\S4]{MilnorDyn} or as a realization of a Berkovich line~\cite{Hrushovski-Loeser-Poonen}. A general dendrite $X$ contains three types of points $x\in X$, according to the number of components of the complement $X\setminus\{x\}$. This number is at most countable and coincides with the Menger--Urysohn \textbf{order} of $x$ in $X$.

\begin{itemize}
\item If the complement $X\setminus\{x\}$ remains connected, $x$ is an \textbf{end}.
\item If $x$ separates $X$ into two components, it is a \textbf{regular point}.
\item Otherwise, $x$ belongs to the set $\Br(X)$ of \textbf{branch points}.
\end{itemize}

The set of branch points is at most countable and provides us with a natural representation
$$\Homeo(X) \longrightarrow \Sym(\Br(X))$$
into the Polish group of permutations of $\Br(X)$. When $X$ has no \textbf{free arcs}, i.e.\ when $\Br(X)$ is arc-wise dense in $X$, this representation is a topological group isomorphism onto its image, see~\cite[2.4]{DM_structure}.

We shall focus on certain universal dendrites $D_n$, namely the \textbf{Wa\.zewski dendrite} of order $3\leq n\leq\infty$. This dendrite is characterized up to homeomorphism by the following properties:

\begin{itemize}
\item $D_n$ is not reduced to a single point,
\item every branch point of $D_n$ has order $n$,
\item $D_n$ has no free arcs.
\end{itemize}

\noindent
Consider now an arbitrary subgroup $\sG<\Homeo(D_n)$. On the one hand, $\sG$ is a permutation group of the countable set $\Br(D_n)$. On the other hand, for any $x\in\Br(D_n)$, the stabilizer $\sG_x$ projects to a permutation group $\Gamma(\sG_x)$ of the $n$-element set of components
$$\comp x := \pi_0\Big(D_n\setminus\{x\}\Big).$$
In particular, provided the permutation group $\sG$ is transitive on branch points, the permutation groups $(\sG_x,\comp x)$, as $x\in \Br(D_n)$ varies, are isomorphic, so $\sG$ determines a well-defined isomorphism type of a permutation group $\Gamma <\Sym(n)$. One thus obtains a functor which produces from a branch point transitive $\sG<\Homeo(D_n)$ a permutation group $(\Gamma, [n])$. The permutation group $(\Gamma, [n])$ is 
called the \textbf{local action} of $\sG$.

The kaleidoscopic construction aims at reversing the functor $\sG\mapsto (\Gamma,[n])$. Thus, we start with a permutation group $\Gamma$ of the set $[n]$. We seek a canonical subgroup $\Uf(\Gamma) < \Homeo(D_n)$ (up to conjugation) whose stabilizers $\Uf(\Gamma)_x$ act on $\comp x$ like $\Gamma$ acts on $[n]$. This requires an identification of each $\comp x$ with $[n]$, that is, a \textbf{coloring}. Given a coloring, we can define a subgroup of $\Homeo(D_n)$ by postulating that the coloring is only locally changed by elements of $\Gamma$, in close analogy to the definition proposed by Burger--Mozes~\cite{Burger-Mozes1} for tree automorphisms.

In order for this construction to be well-defined and canonical, it turns out that we cannot use just any coloring. We shall need \textbf{kaleidoscopic colorings}, defined by the following requirement: given any two points $x\neq y$ and any two colors $i\neq j$, there is a branch point separating $x$ from $y$ at which the component of $x$ has color $i$ and the component of $y$ has color $j$. We shall prove that kaleidoscopic colorings are unique up to homeomorphisms, which ensures that $\Uf(\Gamma)$ is well-defined as an isomorphism type of permutation groups. It may not be obvious at first that kaleidoscopic colorings exist at all, but in fact, much more is true: we shall see that they form a dense \gdelta in the space of all colorings. This makes the construction generic, and hence canonical.

\smallskip
\begin{center}
* \kern5mm *  \kern5mm *
\end{center}

\smallskip
We now turn to further properties of the kaleidoscopic functor ${\Gamma \mapsto \Uf(\Gamma)}$. Recall that a permutation group is \textbf{oligomorphic} if for all $k\in\Nb$ the diagonal action on $k$-tuples has finitely many orbits. This condition is of course automatic for $\Gamma < \Sym(n)$ if $n < \infty$, but it is a very interesting property for infinite permutation groups~\cite{Cameron_oligo}, such as for instance $\Uf(\Gamma)$.

\begin{thm}\label{thm:oligo:intro}
The permutation group $\Uf(\Gamma)$ is oligomorphic if and only if $\Gamma$ is oligomorphic.
\end{thm}

We have already stated that $\Uf(\Gamma)$ is always a perfect group; equivalently, that the homology $\HH_1(\Uf(\Gamma),\Zb)$ vanishes. This raises the question of higher homological finiteness properties, starting with the Schur multiplier $\HH_2(\Uf(\Gamma),\Zb)$. Interestingly, the situation now depends on $\Gamma$. For instance, the full symmetric group $\Gamma=\Sym(\infty)$ corresponds to $\Uf(\Gamma)=\Homeo(D_\infty)$, and the latter has no homology at all:

\begin{thm}\label{thm:acyclicity:intro}
The group $\Homeo(D_\infty)$ is acyclic, that is, the homology groups $\HH_n(\Homeo(D_\infty),\Zb)$ vanish for all $n>0$.
\end{thm}

In contrast, there are many situations where $\Uf(\Gamma)$ has a non-trivial Schur multiplier. Following P.~Neumann~\cite{Neumann75}, a permutation group is called \textbf{generously transitive} if it can transpose any pair of points. We shall further say that an action is \textbf{semi-generous} if it decomposes into two orthogonal orbits that are both generously transitive, recalling that two actions are \textbf{orthogonal} if the diagonal action on the product is transitive.

As soon as we are \emph{not} in these situations, we obtain non-trivial cohomology $\HH^2$ and indeed also non-trivial bounded cohomology $\HB^2$.

\begin{thm}\label{thm:generous:intro}
Suppose that the permutation group $\Gamma$ is neither generous nor semi-generous. Then $\Uf(\Gamma)$ admits a cocycle which determines a non-trivial class in $\HH^2(\Uf(\Gamma),\Zb)$ and in $\HB^2(\Uf(\Gamma),\Zb)$. Moreover, these classes remain non-trivial when viewed as $\Rb$-valued cohomology classes.
\end{thm}

This shows that kaleidoscopic groups, despite their dendritic nature, should not be considered  as one-dimensional objects in any strict sense.

\smallskip
Since $\Uf(\Gamma)$ is perfect, the universal coefficient theorem allows us to deduce that the Schur multiplier is often non-trivial.

\begin{cor}
Suppose that the permutation group $\Gamma$ is neither generous nor semi-generous. Then the Schur multiplier $\HH_2(\Uf(\Gamma),\Zb)$ is non-trivial, indeed non-torsion.\qed
\end{cor}

Furthermore, in the particular case of \emph{trivial} local actions, we obtain kaleidoscopic groups with large Schur multipliers. The case $n=\infty$ of the next statement is in stark contrast to the acyclicity result of Theorem~\ref{thm:acyclicity:intro}.

\begin{cor}\label{cor:schur:large:intro}
The Schur multiplier of the kaleidoscopic group $\Uf(1\acts [n])$ has rank at least $(n-1)(n-2)/2$.
\end{cor}

We shall see in Section~\ref{sec:coho} that the restriction on (semi-)generosity in Theorem~\ref{thm:generous:intro} is in fact the exact condition needed for our method of proof to work when $n<\infty$. However, when $n=\infty$, we will give an example of a generously transitive permutation group $\Gamma$ whose associated kaleidoscopic group has non-trivial Schur multiplier.

Providing even more contrast to the acyclicity of Theorem~\ref{thm:acyclicity:intro}, we exhibit in Corollary~\ref{cor:T:hom} a kaleidoscopic group with non-trivial homology in every even degree.

\medskip
As mentioned above, there is a natural topology on the group $\Uf(\Gamma)$ since it is constructed as a subgroup of the Polish group $\Sym(\Br(X))$. This topology turns $\Uf(\Gamma)$ into a Polish group exactly when it is closed. We know precisely when the latter happens:

\begin{thm}
The group $\Uf(\Gamma)$ is closed in $\Sym(\Br(X))$ if and only if $\Gamma$ is closed in $\Sym(n)$, which is automatic when $n<\infty$.
\end{thm}

There are moreover many cases where this is the \emph{only} possible structure of a Polish group on the abstract group $\Uf(\Gamma)$.

\begin{thm}
If $\Gamma$ is discrete in $\Sym(n)$, then there is a unique topology on $\Uf(\Gamma)$ for which it is a Polish group.
\end{thm}

The discreteness assumption is automatic when $n<\infty$, and it is also satisfied, for example, in the continuum of examples arising from regular actions $\Gamma\acts\Gamma$ as mentioned above.

\medskip

Reflecting again the similarity with the groups constructed by Burger--Mozes~\cite{Burger-Mozes1}, we also obtain a universal property under the strong, but necessary hypothesis of double transitivity.

\begin{thm}
For $3\leq n \leq \infty$, let $G\leq \Homeo(D_n)$ be a subgroup which is transitive on branch points and  $\Gamma\acts[n]$ be the isomorphism type of its local action. If $\Gamma$ is doubly transitive, then $G$ is contained in $\Uf(\Gamma)$ (for a suitable kaleidoscopic coloring).
\end{thm}

%
%
%

\bigskip
\subsection*{Location of the proofs} Theorem~\ref{thm:general} is a patchwork. Simplicity is proved in Theorem~\ref{thm:simple} and uniform perfectness in  Theorem~\ref{thm:uniform_perfect}. Transitivity properties can be found in Corollary~\ref{cor:U(F)_trans} and Theorem~\ref{thm:double_transitivity}. The wreath product structure of stabilizers is recorded in Corollary~\ref{cor:wreath}.

Regarding non-isomorphisms, Theorem~\ref{thm:intro:isom} is ensured by Theorem~\ref{thm:isom}. The characterization of oligomorphy is  Theorem~\ref{thm:oligo}. All non-vanishing results for (co)homology are in Section~\ref{sec:coho}. The acyclicity of $\Homeo(D_\infty)$ is Theorem~\ref{thm:acyclic}. The characterization of Polishness of $\Uf(\Gamma)$ is Proposition~\ref{prop:closed subspace top}. The uniqueness of the topology is proved in Theorem~\ref{thm:unique_polish_top}. Universality of $\Uf(\Gamma)$ is Theorem~\ref{thm:universal}.

\setcounter{tocdepth}{1}
\tableofcontents

\section{Preliminaries}
\subsection{General notation}
We write $\Nbe=\mathbf{N}\cup \{\infty\}$, where $\Nb$ contains~$0$ and $\infty$ abusively denotes both the first infinite ordinal and its cardinal. For $n\in\Nbe$, we set $[n]=\{m\in\Nb: m<n\}$ and define $\Sym(n)$ to be the symmetric group of $[n]$ endowed with the topology of pointwise convergence; this topology is discrete when $n\neq\infty$. We further write $\Nbe_{\geq k}=\{n\in\Nbe : n\geq k\}$ when $k\in\Nb$.

For any set $X$ and $k\in \Nb$, we denote by $X^{(k)} \se X^k$ the subset of \emph{distinct} $k$-tuples.

Given a group $G$ acting on a set $X$ and $Y\subseteq X$,  we write $G_{(Y)}$ for the pointwise stabilizer of $Y$ in $G$.

\subsection{Background on dendrites}
A topological space $X$ is called a \textbf{dendrite} if it is a locally connected continuum containing no simple closed curve. A variety of other equivalent characterizations can be found in~\cite[\S10]{Nadler}, to which we refer for detailed background. A fundamental property of dendrites is that the intersection of any two connected subsets of $X$ remains connected. It follows that every non-empty subset $Y\se X$ is contained in a unique minimal closed connected subset of $X$. This subset is itself a dendrite, and we denote it by $[Y]$. When $Y$ contains just two points $x,y$ we also write $[x,y]$ for this sub-dendrite, which is a topological \textbf{arc} unless $x=y$. Any two distinct points of $X$ are thus joined by a unique arc, and we write $(x,y)$ for $[x,y]\setminus \{x,y\}$ when no confusion with the pair $(x,y)$ is to be feared. That points are connected by unique arcs gives rise to a canonical ternary relation $B(x,y,z)$ defined by $B(x,y,z)\Leftrightarrow y\in (x,z)$.  The relation $B(x,y,z)$ is called the \textbf{betweeness} relation. Equivalently, we say that  $y$ \textbf{separates} $x$ and $z$.

Given a point $x$ in a dendrite $X$, we use the shorthand
$$\comp x \ = \ \pi_0(X\setminus \{x\})$$
for the space of components of the complement of $\{x\}$. This is always an at most countable set, and its cardinal coincides with the Menger--Urysohn \textbf{order} of $x$ in $X$. The point $x$ is called an \textbf{end point}, a \textbf{regular point} or a \textbf{branch point} according to whether this order is one, two or at least three. The corresponding subsets of $X$ are denoted respectively by $\Ends(X)$, $\Reg(X)$ and $\Br(X)$. The set $\Br(X)$ is at most countable. When $X$ is not reduced to a point, the set $\Reg(X)$ is arc-wise dense, and we have $X=[\Ends(X)]$. For any $x\neq y$ in $X$, we denote by $U_x(y)$ the component of $\comp x$ that contains $y$.

One can organise the union of all $\comp x$ into a bundle over $X$ (see \S7 in~\cite{DM_dendrites}), but we shall be more interested here in the \textbf{branch bundle} $\widehat{X}$ of $X$ defined to be the union of all $\comp x$ when $x$ ranges over the branch points of $X$ only, that is
$$\widehat{X}=\bigcup_{x\in\Br(X)}\widehat{x}.$$
The set $\widehat{X}$ is countable, and we shall not endow it with any topology. 

Since the smallest sub-dendrite containing a finite union of points (or arcs) is the topological realization of a finite tree, a number of familiar combinatorial concepts and arguments can be adapted to such finite collections. For a finite set $F\subseteq D$, we say that distinct $f,g\in F$ are \textbf{adjacent} if $(f,g)\cap F=\emptyset$. Given any three distinct points $x,y,z\in X$, we define the \textbf{center} $\kappa(x,y,z)$ of $\{x,y,z\}$ as the unique point in $[x,y]\cap [y,z]\cap [z,x]$. A subset of $X$ will be called \textbf{center-closed} if it contains the center of any three distinct points in it. The following lemma can be immediately reduced to a corresponding statement about finite trees, which is elementary.

\begin{lem}\label{lem:centers}
Let $F\se X$ be a finite subset and let $C$ be the collection of the centers of all triples of distinct points in $F$. Then the finite set $F\cup C$ is center-closed.\qed
\end{lem}

Given a sub-dendrite $Y\se X$, there exists a canonical continuous retraction $r\colon X\to Y$ called the \textbf{first-point map} satisfying $[r(x), x] \cap Y = \{r(x)\}$ for all $x\in X$.

\begin{lem}\label{lem:first pt map}
If $Y=[F]$ for a non-empty compact subset $F\se \Br(X)$, then $r(x)\in\Br(X)$ for all $x\notin Y$.
\end{lem}

\begin{proof}
Since $Y$ is itself a dendrite, we can consider $\Ends(Y)$. We have that $\Ends(Y)\se F$ by \cite[Lemma 2.3]{DM_dendrites} applied to $Y$, and hence $\Ends(Y)\se \Br(X)$. It remains only to consider the case where $r(x)$ lies inside an arc of $Y$, and we conclude that it is a branch point of $X$ since $[r(x), x] \cap Y = \{r(x)\}$.
\end{proof}

\subsection{Patchwork}
One reason for the flexibility of dendrite groups is that it is possible to patch together homeomorphisms. An example of such a statement is \cite[Lemma 2.9]{DM_structure}, and we shall need the following strengthening of that lemma.

\begin{lem}\label{lem:patchwork2}
Let $\sU$ be a family of disjoint open subsets of a dendrite $X$. For each $U\in \sU$, let $f_U$ be a homeomorphism of $X$ which is the identity outside $U$. Then the map $f\colon X\to X$ given by $f_U$ on each $U\in \sU$ and the identity elsewhere is a homeomorphism.
\end{lem}

Since Lemma~\ref{lem:patchwork2} does not seem to reduce immediately to~\cite[Lemma 2.9]{DM_structure}, which assumes the sets $U$ are connected, we give a full proof.

\begin{proof}
Since $f$ is well-defined and bijective, it suffices to prove sequential continuity. Suppose for a contradiction that there is a sequence $x_n$ converging to some $x$ with $f(x_n)$ not converging to $f(x)$. By extracting, we can assume that $f(x_n)$ converges to some $y\neq f(x)$. We can further assume that each $x_n$ belongs to some $U_n\in \sU$ since otherwise $y=x=f(x)$. Each $U\in\sU$ contains $x_n$ for at most finitely many indices $n$ since otherwise $y=f_U(x)=f(x)$. In particular, $x$ cannot belong to any $U$, and thus it is fixed by all $f_U$.

The same reasoning shows that $y$ is also fixed by all $f_U$, and thus the arc $[x,y]$ is preserved by all $f_U$. In particular, the  first-point map $r\colon X\to [x,y]$ is $f_U$-equivariant. We define $x'_n=r(x_n)$ and $y'_n=r(f(x_n))$ and note that these sequences converge to $x$ and $y$ respectively.

Let $z$ be any point separating $x$ and $y$. Since the various $U\in \sU$ are disjoint, we have $z\notin U_n$ when $n$ is large enough; thus $z$ is $f_{U_n}$-fixed for such $n$. We have moreover $x'_n\in [x, z)$ and $y'_n\in (z,y]$ for $n$ large. Using equivariance, we have $f_{U_n}([x, x'_n]) = [x, y'_n]$. This, however, is impossible since the latter arc contains the fixed point $z$ but the former does not.
\end{proof}

\subsection{Universal dendrites}\label{sec:Waz}
For every $n\in \Nbe_{\geq 3}$ there exists a dendrite $D_n$, the \textbf{Wa\.zewski dendrite} of order $n$, with the following properties:

\begin{enumerate}[(a)]
\item $D_n$ is not reduced to a single point,
\item every branch point of $D_n$ has order $n$ and
\item $\Br(D_n)$ is arc-wise dense in $D_n$.
\end{enumerate}

These properties determine $D_n$ uniquely up to homeomorphisms, and moreover, $D_n$ is universal in the sense that every dendrite whose points all have order~$\leq n$ can be embedded into $D_n$. We refer again to~\cite[\S10]{Nadler}; the original construction of $D_\infty$ is due to Wa\.zewski~\cite{WazewskiPHD}, \cite{Wazewski23}. The uniqueness of $D_n$ is proved in a more general setting in~\cite[6.2]{Charatonik-Dilks}.

The topology of the homeomorphism group of Wa\.zewski dendrites is given by the permutation topology on branch points. Indeed, the map $\Homeo(D_n)\rightarrow \Sym(\Br(D_n))$ is a topological group isomorphism onto its image; see \cite[Proposition 2.4]{DM_structure}. We shall often appeal to this fact implicitly. As a consequence, we can obtain homeomorphisms between sub-dendrites of $D_n$ by considering branch points. Recall that for any dendrite $X$, the betweeness relation $B(x,y,z)$ holds if and only if $y\in (x,z)$.  Observe that a dendrite has no free arc if and only if any two distinct points are separated by some branch point.

\begin{prop}\label{prop:lifting}
Let $X,Y$ be dendrites without free arc. If $g \colon\Br(X)\to\Br(Y)$ is a bijection that preserves the betweeness relation, then there is a homeomorphism $h\colon X\to Y$ extending $g$.
\end{prop}

\begin{proof}
Given $b\in\Br(X)$, we have an equivalence relation on $\Br(X)\setminus \{b\}$ by declaring that points \emph{not} separated by $b$ are equivalent. In other words, the equivalence classes are the sets of branch points in some connected component of $X\setminus\{b\}$. Since $g$ preserves the betweeness relation, it intertwines this relation with the corresponding relation induced by $g(b)$ on $\Br(Y)\setminus \{g(b)\}$. This allows us to define a map $\widehat{g }\colon\widehat{X}\to\widehat{Y}$. This map is a bijection and its restriction to $\widehat{b}$ is also a bijection onto $\widehat{g (b)}$.

Now, for $x\in X$, we define $h (x)=\bigcap_C\overline{\widehat{g }(C)}$, where $C$ ranges over all components $C\in\widehat{X}$ that contain $x$. By compactness, this intersection is non-empty as soon as any finite sub-intersection is non-empty. This is the case because, if $C_1,\dots,C_n$ contain $x$, they also contain some branch point $b$ and thus $\widehat{g }(C_1)\cap \cdots\cap \widehat{g }(C_n)$ contains $g (b)$. Moreover, this intersection is reduced to a point because any two distinct points are separated by some branch point. The map $h \colon X\to Y$ is thus well defined and by construction coincides with $g $ on $\Br(X)$. It has an inverse, namely the map constructed in the same way with respect to $g^{-1}$. 

It remains to show that $h$ is continuous or equivalently open. By construction, it maps connected components of complement of points to connected components of complement of points. Since these open sets generate the topology~\cite[2.10]{DM_structure}, we conclude that $h$ is open.
\end{proof}

\section{Defining kaleidoscopic groups}

\subsection{Colorings}

\begin{defn}
A \textbf{coloring} of the dendrite $D_n$ is a map $c\colon\widehat{D_n}\rightarrow [n]$ such that the restriction $c\rest_{\comp x}$ at any branch point $x\in \Br(D_n)$ is a bijection $c\rest_{\comp x}\colon \comp x \xrightarrow{\,\cong\,} [n]$. We define $c_x\colon X\setminus\{x\} \to [n]$ by $c_x(y) = c(U_x(y))$, where $U_x(y)$ is the element of $\comp x$ that contains $y$.  We often abuse notation and also write $c_x$ for $c\rest_{\comp x}$, since $c_x$ is constant on the components $\comp x$.
\end{defn}

Under the product topology, $[n]^{\widehat{D_n}}$ is a Polish space, since $\widehat{D_n}$ is countable. The set of colorings $\mc{C}\subseteq [n]^{\widehat{D_n}}$ is easily verified to be a $G_{\delta}$ set, so that it is a Polish space under the subspace topology; see for instance~\cite[Thm.~3.11]{Kechris95}. When $[n]$ is finite, $\mc{C}$ is in fact closed.

\begin{defn}
A coloring of $D_n$ is \textbf{kaleidoscopic} if for all distinct $x,y\in \Br(D_n)$ and all distinct $i, j\in [n]$, there is $z\in (x,y)\cap \Br(D_n)$ such that $c_z(x)=i$ and $c_z(y)=j$.  
\end{defn}
In the following picture, open alcoves depict the two components $U_z(x)$ and $U_z(y)$ respectively colored with $i$  and $j$.
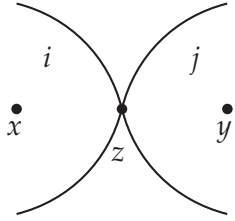
\begin{center}
\setlength{\unitlength}{0.7cm}
\thicklines
\begin{picture}(4,4)
\cbezier(2,2)(2.25,3)(3,3.75)(4,4)
\cbezier(2,2)(2.25,1)(3,.25)(4,0)
\cbezier(2,2)(1.75,3)(1,3.75)(0,4)
\cbezier(2,2)(1.75,1)(1,.25)(0,0)
\put(2,2){\circle*{0.2}}\put(1.8,1){$z$}
\put(0,2){\circle*{0.2}}\put(-.2,1.5){$x$}
\put(4,2){\circle*{0.2}}\put(3.8,1.5){$y$}
\put(.5,2.8){$i$}
\put(3.3,2.8){$j$}
\end{picture}
\end{center}
It is not immediately obvious that kaleidoscopic colorings exist, but much more is in fact true: they are generic.

\begin{prop}
For any $n\in \Nbee$, the set of kaleidoscopic colorings is a dense \gdelta in the space of all colorings.
\end{prop}

\begin{proof}
Since the set of branch points is countable, it suffices by Baire's theorem to show that for any branch points $x\neq y$ and colors $i\neq j$, the set of colorings for which there is $z$ such that $c_z(x)=i$ and $c_z(y)=j$ is open and dense. It is open because for every $z$ in $(x,y)\cap \Br(D_n)$, the conditions $c_z(x)=i$ and $c_z(y)=j$ are open. By definition of the pointwise convergence topology, the density follows from the fact that the set $(x,y)\cap \Br(D_n)$ is infinite.
\end{proof}

We now argue that all kaleidoscopic colorings differ by an element of $\Homeo(D_n)$, seeing $g\in \Homeo(D_n)$ as acting on $\widehat{D_n}$. We prove the following stronger statement, because it will be needed later in this text.

\begin{thm}\label{thm:colors}
Let $n\in \Nbee$, $X,Y$ be dendrites homeomorphic to $D_n$ and $c$ and $d$ be kaleidoscopic colorings of $X$ and $Y$, respectively. Then there exists a homeomorphism $h\colon X\to Y$ such that $d\circ h=c$.

Furthermore, let $e_0, e_1\in X$ be distinct end points and likewise $f_0, f_1\in Y$. Let $x\in [e_0, e_1]$ and $y\in [f_0, f_1]$ be branch points with $c_x(e_i)=d_y(f_i)$ for $i=0,1$. Then $h$ can be chosen such that $h(e_i)=f_i$ for $i=0,1$ and such that $h(x)=y$.
\end{thm}

The proof will use the back-and-forth technique similar to~\cite[6.2]{Charatonik-Dilks}. We start by recording the following classical example of this technique, contained in \cite[Lemma 6.1]{Charatonik-Dilks}.

\begin{lem}\label{lem:BF:basic}
Let $A_j, B_j$ be sequences of countable dense subsets of the interval $(0,1)$, where $j$ ranges over a finite or countable index set. Suppose that all $A_j$ are pairwise disjoint and likewise for the $B_j$. Then there is an orientation-preserving homeomorphism $g$ of $[0,1]$ such that $g(A_j)=B_j$ for all $j$.\qed
\end{lem}

Given any distinct points $x,x'$ in a dendrite (homeomorphic to) $D_n$, a coloring $c$ defines a partition of the set of branch points $z$ in $(x,x')$ according to the pair of values $(c_z(x), c_z(x'))$. If $c$ is kaleidoscopic, then each block of this partition is dense in the arc $[x,x']$. Moreover, the value $c_z(\cdot)$ remains constant on $(z,x]$ and on $(z,x']$. Therefore, Lemma~\ref{lem:BF:basic} implies the following.

\begin{lem}\label{lem:BF:arc}
Let $X$, $Y$, $c$ and $d$ be as in Theorem~\ref{thm:colors}. Given any distinct $x, x'\in X$ and distinct $y, y'\in Y$, there is a homeomorphism $g\colon [x, x']\to [y, y']$ with $g(x)=y$, sending $(x, x')\cap \Br(X)$ onto $(y, y')\cap \Br(Y)$ and such that $d_{g(z)}(g(t)) = c_z(t)$ holds for all $z\in (x, x')\cap \Br(X)$ and all $t\in  [x, x']\setminus \{z\}$.\qed
\end{lem}

\begin{proof}[Proof of Theorem~\ref{thm:colors}]
We shall construct inductively two increasing sequences of sub-dendrites $X_m\se X$ and $Y_m\se Y$ and homeomorphisms $g_m\colon X_m\to Y_m$ such that $g_{m+1}$ extends $g_m$, such that $g_m$ maps $X_m\cap \Br(X)$ onto $Y_m\cap \Br(Y)$ and such that $d\circ g_m=c$ on $X_m$.

First, extend $e_0, e_1$ and $f_0, f_1$ to dense sequences $(e_i)_{i\in\Nb}$ in $\Ends(X)$ and $(f_i)_{i\in\Nb}$ in $\Ends(Y)$; recall that the end points are dense in $D_n$ \cite[2.3]{Charatonik91}.

For the base case $m=0$, let $X_0=\{x\}$, $Y_0=\{y\}$ and $g_0\colon x\mapsto y$.

The inductive step, when $X_m$, $Y_m$ and $g_m$ are already constructed, depends on the parity of $m$. We start with $m$ even. Let $j\geq 0$ be the smallest index for which $e_j\notin X_m$ and define $X_{m+1}=[X_m\cup\{e_j\}]$. The image $x_m$ of $e_j$ under the first point retraction $X\to X_m$ is a branch point of $X$, since by construction all $X_m$ are spanned by $x$ and end points. The image $y_m:=g_m(x_m)$ is thus a branch point of $Y$ as well. By the induction hypothesis, the color at $x$ of elements of $X_m$ coincides with the corresponding color at $y$ of their image in $Y_m$. Therefore, since $e_j$ has a color not seen in $X_m$ from $x$, there is a component at $y$ not meeting $Y_m$ of that same color. By density of $(f_i)_{i\in\Nb}$, we can choose an index $j'$ with $f_{j'}$ in that component. We define $Y_{m+1}=[Y_m\cup\{f_{j'}\}]$ and extend $g_m$ to a homeomorphism $g_{m+1}$ by applying Lemma~\ref{lem:BF:arc} to the arcs $[x_m,e_j]$ and $[y_m,f_{j'}]$. Notice that $g_{m+1}$ does indeed preserve the colorings on the whole of $X_{m+1}$ and $Y_{m+1}$: for instance, the color of any element of $X_m$ seen from a branch point in $(x_m,e_j)$ is simply the color of $x_m$.

The case of $m$ odd is identical after exchanging $X$ with $Y$ and replacing $g_m$ by its inverse; the inductive construction is complete.

The increasing union of all $X_m$ contains in particular the entire sequence $(e_i)_{i\in\Nb}$ (thanks to all even steps); this implies that it contains all branch points of $X$. The corresponding statement holds for the union of all $Y_m$. Therefore, Proposition~\ref{prop:lifting} shows that the bijection $g$ between these unions obtained from the maps $g_m$ extends to a homeomorphism $h$ with the desired properties.
\end{proof}

\subsection{Kaleidoscopic groups}
\begin{defn}
Let $c$ be a coloring of $D_n$. The \textbf{local action} of $g\in \Homeo(D_n)$ at $x\in \Br(D_n)$ is the element $\sigma_c(g,x)$ of $\Sym([n])$ defined by the cocycle
$$\sigma_c\colon  \Homeo(D_n)\times  \Br(D_n)\longrightarrow \Sym([n]), \kern3mm \sigma_c(g,x):=c_{g(x)}\circ g\circ c_x^{-1}$$
wherein $g$ is considered both as a map on $D_n$ and on $\widehat{D_n}$.
\end{defn}

In particular, we record the corresponding cocycle properties. For $g,h\in \Homeo(D_n)$ and $x\in \Br(D_n)$,
$$\sigma_c(gh,x)=\sigma_c(g,hx)\sigma_c(h,x) \kern3mm\text{and}\kern3mm \sigma_c(g,x)^{-1}=\sigma_c(g^{-1},g(x)).$$

\begin{defn}
Let $c$ be a coloring of $D_n$ for $n\in \Nbee$. For any permutation group $\Gamma\leq \Sym(n)$, the group with local action $\Gamma$ is defined to be
\[
\Uf_c(\Gamma)=\{g\in \Homeo(D_n): \forall x\in \Br(D_n),\; \sigma_c(g,x)\in \Gamma\}.
\]
When $c$ is a kaleidoscopic coloring, we call $\Uf_c(\Gamma)$ a \textbf{kaleidoscopic group} with local action $\Gamma$.
\end{defn}

We emphasize that this construction depends on $\Gamma$ as a permutation group of $[n]$ rather than as an abstract group. For instance, the special case of the trivial group $\Gamma=1$ is far from trivial, as will be seen below. Note that $\Uf_c(1)$ is exactly the subgroup of $\Homeo(D_n)$ that preserves the coloring and $\Uf_c(\Sym([n]))$ is actually $\Homeo(D_n)$. When necessary, we shall highlight the fact that $\Gamma$ stands for a permutation group by using notations such as $\Uf_c(\Gamma\acts[n])$, for instance $\Uf_c(1\acts[n])$.

The group $\Uf_c(\Gamma)$ also depends on the coloring $c$, but for kaleidoscopic colorings, the dependence disappears, as demonstrated by the following consequence of Theorem~\ref{thm:colors}.

\begin{cor}\label{cor:conjugate B-M groups}
For $D_n$ with $n\in \Nbee$ and $\Gamma\leq \Sym(n)$, if $c$ and $d$ are kaleidoscopic colorings of $D_n$, then $\Uf_c(\Gamma)$ and $\Uf_d(\Gamma)$ are conjugate by an element of $\Homeo(D_n)$.
\end{cor}

The isomorphism type of a kaleidoscopic group and of its action on the dendrite is thus independent of the choice of a kaleidoscopic coloring. We therefore suppress the subscript and simply write $\Uf(\Gamma)$. We stress that whenever we write $\Uf(\Gamma)$, we mean the group $\Uf_d(\Gamma)$ for some kaleidoscopic coloring $d$. When we consider the local action for an element $g\in \Uf(\Gamma)$ at some $x\in \Br(D_n)$, we also suppress the coloring and simply write $\sigma(g,x)$.

\begin{proof}[Proof of Corollary~\ref{cor:conjugate B-M groups}]
By Theorem~\ref{thm:colors}, we may find $h\in\Homeo(D_n)$ such that $d\circ h= c$. Consider $hgh^{-1}\in h\Uf_c(\Gamma)h^{-1}$. We see that
\[
\begin{array}{rcl}
\sigma_d(hgh^{-1},h(v)) & = & d_{hg(v)}\circ hgh^{-1} \circ d^{-1}_{h(v)}\\
						& = & (d\circ h)_{g(v)}\circ g\circ (d\circ h)^{-1}_v\\
						&= & c_{g(v)}\circ g\circ c^{-1}_v;
\end{array}
\]
writing out the commutative diagram makes the previous equalities clear. Since $\sigma_c(g,v)\in \Gamma$, we deduce that $\sigma_d(hgh^{-1},h(v))\in \Gamma$. Therefore, 
\[
h\Uf_c(\Gamma)h^{-1}\leq \Uf_d(\Gamma).
\]
The converse inclusion follows by the same argument using that $c\circ h^{-1}=d$.
\end{proof}

\begin{rmk}
Let $S\se\Nbee$. One can define analogously colorings $c$ of Wa\.zewski dendrites $D_S$ by coloring connected components around a point of order $n\in S$ with colors in $[n]$. By fixing a local group $\Gamma_n$ for each $n\in S$ we can define also a group $\Uf_c((\Gamma_n)_{n\in S})$ such that for any point $x$ of order $n$, $\sigma_c(g,x)\in\Gamma_n$. We will not investigate this more general setting in the present text.
\end{rmk}

\section{Assembling homeomorphisms}\label{sec:patchwork}
In this section, we show how to use Theorem~\ref{thm:colors} to build homeomorphisms of $D_n$ while simultaneously controlling the local action.

\subsection{Colorful patchwork}
Our first result gives sufficient conditions under which a family of homeomorphisms between subsets of a dendrite may be patched together into a homeomorphism of the dendrite. Our primary result here is a mild adaption of~\cite[Lemma 2.9]{DM_structure}.  

Suppose that $D$ is a dendrite and $F\subseteq D$ is finite and center-closed. For each $a\in F$, let
\[
\comp{a}_F=\{C\in \comp a: C\cap F=\emptyset\}.
\] 
For $a,b\in D$ distinct, set $C_{a,b}=U_{a}(b)\cap U_{b}(a)$.
\begin{defn}
The set $\Omega_F=\bigcup_{a\in F}\comp{a}_F\cup \{C_{x,y}: x,y\in F\text{ adjacent}\}$ is called the set of \textbf{components} determined by $F$. 
\end{defn}
Each element of $\Omega_F$ is a path connected open subset of $D$ and is disjoint from $F$. Less trivially, $\Omega_F$ is exactly the collection of connected components of $D\setminus F$.

We are now prepared to prove the desired patchwork lemma. This result follows similarly to~\cite[lemma 2.9]{DM_structure}; we give a proof for completeness.
\begin{defn}
For $R$ and $S$ finite subsets of a dendrite and $f:R\rightarrow S$ a bijection, we say that $f$ is a \textbf{partial dendrite morphism} if $f$ respects betweeness and $f$ carries branch points to branch points, regular points to regular points and end points to end points.
\end{defn}

\begin{lem}[{cf.~\cite[Lemma 2.9]{DM_structure}}]\label{lem:gluing_component_maps}
For $D$ a dendrite, suppose that $R$ and $S$ are finite center-closed subsets of $D$ and $f:R\rightarrow S$ is a partial dendrite morphism. Suppose further the following:
\begin{enumerate}[(a)]
\item For each $A\in \Omega_R$, there is a homeomorphism $h_{A}:A\cup R\rightarrow B\cup S$ such that $B\in \Omega_S$ and $h_{A}\rest_R=f$, and
\item Every $B\in \Omega_S$ equals $\im(h_{A})\setminus S$ for some unique $A\in \Omega_R$. 
\end{enumerate}
Then the map $h:D\rightarrow D$ defined by $h(x)=h_A(x)$ when $x\in A\cup R$ with $A\in \Omega_R$ is a homeomorphism of $D$ such that $h\rest_R=f$.
\end{lem}

\begin{proof}
The map $h$ is a well-defined bijection since $\Omega_R\cup \{ R\}$ is a partition. It thus suffices to show that $h$ is continuous. Fix a compatible metric $d$ for $D$ and suppose that $x_j\rightarrow x$. If infinitely many terms $x_{n_j}$ of the sequence lie in some $A\in \Omega_R$, then $x\in \overline{A}\subset A\cup R$. As $A$ is open, a tail of the sequence lies in $A$. For suitably large $n$, we thus have $h(x_{n})=h_A(x_{n})$, and furthermore, $h_A(x_n)\rightarrow h_A(x)=h(x)$. We conclude that $ h(x_n)\rightarrow h(x)$. 

Let us then suppose that each $A\in \Omega_R$ contains only finitely many terms of the sequence $(x_j)_{j\in \Nb}$. It follows that $x\in F$. Let $(A_i)_{i\in \Nb}$ list $\Omega_R$ and $(B_i)_{i\in \Nb}$ list $\Omega_S$, say that $h_i:A_i\cup R\rightarrow B_i\cup S$ and take $\delta=\min\{d(x,f) : f\in F\setminus\{x\}\}$. Fix $0<\epsilon<\delta/2$. We recall the diameter of any sequence of disjoint connected subsets of a dendrite must tend to zero, see e.g.~\cite[V.2.6]{Whyburn_book}. Thus, there is $N$ such that $\diam(A_i),\diam(B_i)<\epsilon$ for all $i\geq N$. Taking $N$ perhaps larger, we may also assume that $d(x_j,x)<\epsilon$ for all $j\geq N$. Since each $A_i$ only contains finitely many terms of the sequence $(x_j)_{j\in \Nb}$, we may find $M\geq N$ such that $x_j\in A_{i(j)}$ for some $i(j)\geq N$ for every $j\geq M$. For any $A\in \Omega_R$ such that $x\notin \ol{A}$, $d(A,x)>\epsilon$, so we additionally have that $x\in \ol{A}_{i(j)}$ for every $j\geq M$. Since $h_{i(j)}$ is homeomorphism, $h(x)=h_{i(j)}(x)$ lies in $\ol{B}_{i(j)}$. Therefore, 
\[
d(h(x_j),h(x))\leq \diam(B_{i(j)})<\epsilon
\]
for all $j\geq M$. We deduce that $h(x_j)\rightarrow h(x)$. The map $h$ is thus continuous.
\end{proof}

\subsection{Back-and-forth all over again}
To produce the homeomorphisms between connected components required for Lemma~\ref{lem:gluing_component_maps}, we can rely on Theorem~\ref{thm:colors} as follows.

\begin{cor}\label{cor:component_maps_1}
Let $c$ be a kaleidoscopic coloring of $D_n$ for some $n\in \Nbee$. Choose $a,b\in D_n$ and components $A\in \comp{a}$, $B\in \comp{b}$. Then there is a homeomorphism $g\colon A\cup\{a\}\rightarrow B\cup \{b\}$ such that for all $x\in A\cap \Br(D_n)$ and $z \in A\cup\{a\}$ distinct, $c_x(z)=c_{g(x)}(g(z))$ holds.
\end{cor}

\begin{proof}
The characterization of $D_n$ up to homeomorphisms (recalled in Section~\ref{sec:Waz}) shows that both $A\cup\{a\}$ and $B\cup \{b\}$ are homeomorphic to $D_n$. Moreover, $a$ is an end of $A\cup\{a\}$ and $b$ is an end of $B\cup \{b\}$. Therefore, the statement follows from Theorem~\ref{thm:colors}.
\end{proof}

Given two points $a,b$ of a dendrite, recall that $C_{a,b}$ denotes $U_{a}(b)\cap U_{b}(a)$. A variation on the previous statement is needed to build homeomorphisms between sets of the form $C_{a,b}$.

\begin{cor}\label{cor:component_maps_2}
Let $c$ be a kaleidoscopic coloring of $D_n$ for some $n\in \Nbee$. Choose two distinct pairs of points $a_0, a_1\in D_n$  and $b_0,b_1\in D_n$. Then there is a homeomorphism
$$g\colon C_{a_0,a_1}\cup \{a_0,a_1\}\longrightarrow C_{b_0,b_1}\cup \{b_0,b_1\}$$
such that $g(a_i)=b_i$ and such that for all $x\in C_{a_0,a_1}\cap\Br(D_n)$ and $z \in C_{a_0,a_1}\cup \{a_0,a_1\}$ distinct, $c_x(z)=c_{g(x)}(g(z))$ holds.
\end{cor}

\begin{proof}
Again, both $C_{a_0,a_1}\cup \{a_0,a_1\}$ and $C_{b_0,b_1}\cup \{b_0,b_1\}$ are homeomorphic to $D_n$ and we can apply Theorem~\ref{thm:colors}.
\end{proof}

\section{First properties}
In this section, we establish basic topological and permutation-group-theoretic properties of kaleidoscopic groups.

\subsection{Local splitting}
Given a coloring $c$ of $D_n$ and $x\in\Br(D_n)$, where $n\in \Nbee$, it will be convenient to introduce the map
$$\Phi_x\colon \Homeo(D_n)\longrightarrow \Sym(n), \kern10mm g\longmapsto\sigma_c(g,x).$$
By definition, if $\Gamma\leq \Sym(n)$, we have $\Uf_c(\Gamma) = \bigcap_{x\in \Br(D_n)}\Phi_x^{-1}(\Gamma)$. 

We stress that the map $\Phi_x$ is not a group homomorphism. However, the restriction of $\Phi_x$ to the stabilizer of $x$ in $\Homeo(D_n)$ is a homomorphism. When the coloring $c$ is kaleidoscopic, more can be said:

\begin{prop}\label{prop:local:split}
For each $x\in\Br(D_n)$, the map $\Phi_x$ restricted to the stabilizer $\Uf(\Gamma)_x$ yields a split-surjective homomorphism $\Uf(\Gamma)_x \to \Gamma$. Moreover, the section $\Gamma\to\Uf(\Gamma)_x $ can be chosen to be uniformly continuous for the uniform structures of point-wise convergence on $[n]$ and on $\Br(D_n)$ respectively.
\end{prop}

\begin{proof}
Fix $x\in \Br(D_n)$ and let $(A_i)_{i\in [n]}$ enumerate $\comp x$ such that $c(A_i)=i$. For each $i\in [n]$, Corollary~\ref{cor:component_maps_1} supplies a homeomorphism
$$h_i\colon A_0\cup \{x\}\longrightarrow A_i\cup \{x\}$$
such that $h_i(x)=x$ and $c_w(z)=c_{h_i(w)}(h_i(z))$  for all $w,z\in A_0$ distinct. Given $\gamma\in \Gamma$ and $i\in [n]$, we thus have a homeomorphism
$$h_{\gamma(i)}h^{-1}_i\colon A_i\cup\{x\}\longrightarrow A_{\gamma(i)}\cup \{x\}.$$
Since $\gamma$ is a permutation of $[n]$, the conditions of Lemma~\ref{lem:gluing_component_maps} are satisfied by the collection of maps $\{h_{\gamma(i)}h^{-1}_i\}_{i\in [n]}$. We therefore obtain a homeomorphism $g_\gamma$ of $D_n$ defined by $g_{\gamma}(y)=h_{\gamma(i)}h^{-1}_i(y)$, where $i$ is such that $y\in A_i$. Furthermore, $g_{\gamma}$ fixes $x$, and
\[
\sigma_c(g_{\gamma},w)=
\begin{cases}
\gamma & \text{ if }w=x,\\
1 & \text{otherwise.}
\end{cases}
\]
In particular, $g_{\gamma}$ belongs to $\Uf(\Gamma)$ and one verifies readily that $\gamma\mapsto g_{\gamma}$ is a homomorphism.

We now turn to the uniform continuity statement. If we want the map $g_{\gamma}$ to fix a finite subset $F\subseteq \Br(D_n)$, it suffices that $\gamma$ fixes a large enough finite subset of $[n]$ to ensure that $g_{\gamma}$ is the identity on all the components of $\comp x$ that contain elements of $F\setminus \{x\}$.
\end{proof}

\subsection{Topology}

\begin{lem}\label{lem:continuous_cocycle}
Let $c$ be a coloring of $D_n$ where $n\in \Nbee$. For each $x\in \Br(D_n)$, the map $\Phi_x$ is continuous on $\Homeo(D_n)$.
\end{lem}

\begin{proof}
Given finite sequences $r=(r_1,\dots,r_i)$ and $s=(s_1,\dots,s_i)$ in $[n]$, let $O_{ r , s }$ be the set of maps $g\in \Sym(n)$ such that $g(r_j)=s_j$ for $1\leq j\leq i$. Recall that the topology on $\Sym(n)$ has a basis consisting of the sets $O_{ r , s }$; when $n$ is finite, this topology is discrete.

Fix now $r$ and $s$ as above. If the set $\Phi_x^{-1}(O_{r,s })$ is empty, it is trivially open. Otherwise, choose $g\in \Phi_x^{-1}(O_{ r , s })$ and, for each $r_j$, select $w_j\in \Br(D_n)$ such that $c_x(w_j)=r_j$. Define
$$F=\{w_j: 1\leq j\leq i\}\cup\{x\}.$$
Whenever $h$ belongs to the point-wise stabilizer $\Homeo(D_n)_{(F)}$, we have $\sigma_c(gh,x)=\sigma_c(g,x)\sigma_c(h,x)$. For each color $r_j$, it follows that $\sigma_c(gh,x)(r_j)=s_i$, and hence 
\[
g. \Homeo(D_n)_{(F)}\subseteq \Phi_x^{-1}(O_{ r , s }). 
\]
Since $\Homeo(D_n)_{(F)}$ is open, we conclude that $\Phi_x^{-1}(O_{ r , s })$ contains a neighborhood of $g$, so $\Phi_x$ is continuous.
\end{proof}

We now identify exactly when $\Uf(\Gamma)$ is a Polish group under the subspace topology. Recall that this is equivalent to $\Uf(\Gamma)$ being a \gdelta in the Polish group $\Sym(\Br(D_n))$, which in turn is equivalent to being closed; see~\cite[Theorem~3.11 and Exercise~9.6]{Kechris95}.

\begin{prop}\label{prop:closed subspace top} 
Let $\Gamma\leq\Sym(n)$ with $n\in \Nbee$. Then $\Uf(\Gamma)$ is closed in $\Homeo(D_n)$ if and only if $\Gamma$ is a closed subgroup of $\Sym(n)$.
\end{prop}

\begin{proof}
Fix a kaleidoscopic coloring $c$ for $D_n$ and set $G=\Uf_c(\Gamma)$.  Suppose first that $\Gamma$ is a closed subgroup and let $v\in \Br(D_n)$. By Lemma~\ref{lem:continuous_cocycle}, $\Phi_v$ is continuous, so $\Phi_v^{-1}(\Gamma)$ is closed. Since $G=\bigcap_{v\in \Br(D_n)}\Phi_v^{-1}(\Gamma)$, we conclude that $G$ is closed.

Conversely, suppose that $G$ is closed. Fix $x\in \Br(D_n)$ and let $\gamma\mapsto g_{\gamma}$ be a splitting $\Gamma\to G_x$ as in Proposition~\ref{prop:local:split}. The uniform continuity statement implies that $\gamma\mapsto g_{\gamma}$ extends to the closure $\ol\Gamma $ of $\Gamma$ while ranging in $G_x$, since $G$ and hence also $G_x$ is complete. However, this map $\ol\Gamma\to G_x$ is a splitting of $G_x\to \Gamma$, and hence  $\ol\Gamma =\Gamma$ as desired.
\end{proof}

\begin{rmk} 
In view of Proposition~\ref{prop:closed subspace top}, a natural question arises: \textit{Does $\Uf(\Gamma)$ admit a Polish group topology if and only if $\Gamma$ is closed?}
\end{rmk}

\subsection{Transitivity}
A group action $G\acts X$ is called \textbf{doubly transitive} if for any two pairs of distinct elements $(x,y)$ and $(x',y')$ of $X$, there is $g\in G$ such that $g(x)=x'$ and $g(y)=y'$. 

\begin{prop}\label{prop:U(1)_trans}
For $n\in \Nbee$, the group $\Uf(1\acts [n])$ acts transitively on $\Br(D_n)$, and it acts doubly transitively on $\mathrm{Ends}(D_n)$ and on $\mathrm{Reg}(D_n)$.
\end{prop}

\begin{proof}
Theorem~\ref{thm:colors} shows that $\Uf(1\acts [n])$ acts doubly transitively on $\mathrm{Ends}(D_n)$. 

Let $(x,y)$ and $(x^\prime,y^\prime)$ be two pairs of distinct regular points and let us decompose $D_n\setminus\{x,y\}$ into the connected components determined by $\{x,y\}$. That is, $D_n\setminus\{x,y\}=U_x\cup C_{x,y}\cup U_y$ where $C_{x,y}=U_x(y)\cap U_y(x)$, $U_x$ is the unique element of $\comp x$ disjoint from $y$ and $U_y$ is the unique element of $\comp{y}$ disjoint from $x$. Similarly, we write $D_n\setminus\{x^\prime,y^\prime\}=U_{x\prime}\cup C_{x^\prime,y^\prime}\cup U_{y^\prime}$. Appealing to Corollary~\ref{cor:component_maps_1}, there is a color-preserving homeomorphism $h_x\colon U_x\cup \{x,y\}\rightarrow U_{x'}\cup \{x',y'\}$ such that $h_x(x)=x'$ and $h_x(y)=y'$. A second application of Corollary~\ref{cor:component_maps_1} gives a color-preserving homeomorphism $h_y\colon U_y\cup \{x,y\}\rightarrow U_{y'}\cup \{x',y'\}$ such that $h_y(x)=x'$ and $h_y(y)=y'$. Finally, Corollary~\ref{cor:component_maps_2} gives a color-preserving homeomorphism $k\colon C_{x,y}\cup \{x,y\}\rightarrow C_{x',y'}\cup \{x',y'\}$ such that $k(x)=x'$ and $k(y)=y'$. By Lemma~\ref{lem:gluing_component_maps}, we can patch together these homeomorphisms to get a global homeomorphism $f$ that preserves the coloring and such that $f(x)=x'$ and $f(y)=y'$. This proves that $\Uf(1\acts [n])$ acts doubly transitively on $\Reg(D_n)$.

Turning to branch points, let $a, b \in \Br(D_n)$. For each $A\in\comp a$, there is exactly one $B\in \comp b$ such that $c(A)=c(B)$. Corollary~\ref{cor:component_maps_1} supplies a homeomorphism $g_A:A\cup \{a\}\rightarrow B\cup \{b\}$ such that $g_A(a)=b$ and $c_x(z)=c_{g_A(x)}(g_A(z))$ for all $x, z \in A\cup \{a\}$ distinct. The family of maps $(g_A)_{A\in \comp a}$ satisfies the conditions of Lemma~\ref{lem:gluing_component_maps}, so we can patch together these homeomorphisms to get a global homeomorphism $g$. It follows that $g$ is an element of $\Uf(1\acts [n])$, and we conclude that $\Uf(1\acts [n])$ acts transitively on $\Br(D_n)$. 
\end{proof}

Since $\Uf(\Gamma)$ contains $\Uf(1\acts [n])$ for any $\Gamma\leq \Sym(n)$, we immediately obtain the following corollary.

\begin{cor}\label{cor:U(F)_trans}
For $n\in \Nbee$ and $\Gamma\leq \Sym(n)$, the group $\Uf(\Gamma)$ acts transitively on $\Br(D_n)$, and it acts doubly transitively on $\mathrm{Ends}(D_n)$ and on $\mathrm{Reg}(D_n)$. In particular, $\Uf(\Gamma)_{\xi}$ has uncountable index in $\Uf(\Gamma)$ where $\xi\in \mathrm{Ends}(D_n)\cup \mathrm{Reg}(D_n)$.
\end{cor}

One naturally asks when the action $\Uf(\Gamma)\acts \Br(D_n)$ is doubly transitive. Our next theorem characterizes when this phenomenon occurs.

\begin{thm}\label{thm:double_transitivity}
Say that $n\in \Nbe_{\geq 3}$ and $\Gamma\leq \Sym(n)$. Then $\Gamma\acts [n]$ transitively if and only if $\Uf(\Gamma)$ acts doubly transitively on $\Br(D_n)$.
\end{thm}

\begin{proof} Fix $x\neq y\in\Br(D_n)$  and $x^\prime\neq y^\prime\in\Br(D_n)$ and choose $\gamma_x$ and $\gamma_y$ in $\Gamma$ such $\gamma_x (c_x(y))=c_{x^\prime}(y^\prime)$ and $\gamma_y (c_y(x))=c_{y^\prime}(x^\prime)$. For each $i\in[n]$ and $z\in\Br(D_n)$, let us denote by $U_{z,i}\in\wh{z}$ the component such that $c(U_{z,i})=i$. 

For $i\neq c_x(y)$, Corollary~\ref{cor:component_maps_1} supplies a color-preserving homeomorphism $f_{x,i}\colon U_{x,i}\cup \{x,y\}\to U_{x^\prime,\gamma_x (i)}\cup\{x',y'\}$ such that $f_{x,i}(x)=x'$ and $f_{x,i}(y)=y'$.  Similarly, for $i\neq c_y(x)$, we obtain a color-preserving homomorphism $f_{y,i}\colon U_{y,i}\to U_{y^\prime,\gamma_y (i)}$ such that $f_{y,i}(x)=x'$ and $f_{y,i}(y)=y’$. Finally, Corollary~\ref{cor:component_maps_2} provides a color-preserving map $f_{x,y}\colon C_{x,y}\cup \{x,y\}\to C_{x^\prime,y^\prime}\cup\{x',y'\}$ that sends $x$ to $x^\prime$ and $y$ to $y^\prime$.  By Lemma~\ref{lem:gluing_component_maps}, we can patch together  these homeomorphisms to get a global homeomorphism $f$ of $D_n$ sending $x$ to $x^\prime$, $y$ to $y^\prime$, such that $\sigma_c(f,x)=\gamma_x$, $\sigma_c(f,y)=\gamma_y$ and $\sigma_c(f,z)=\id$ for any $z\neq x,y$. We deduce that $f\in \Uf(\Gamma)$, and thus, $\Uf(\Gamma)$ acts doubly transitively on $\Br(D_n)$. 

The converse is immediate.
\end{proof}

\subsection{Oligomorphy}
For a group $G$ acting on a set $X$, the \textbf{diagonal action} of $G$ on $X^k$ with $k\geq 1$ is defined by $g((x_1,\dots,x_k))=(g(x_1),\dots,g(x_k))$. The orbits of the diagonal action of $\Homeo(D_n)$ on $\Br(D_n)^k$ are completely determined by the combinatorial tree that is formed from the elements of the $k$-tuple; see Proposition~6.1 in~\cite{DM_structure}. The following result describes orbits of the diagonal action of $\Uf(\Gamma)$ on $\Br(D_n)^k$. Here one has to consider not only the type of tree generated but also the orbits of diagonal action of $\Gamma$ on $[n]^m$ for $1\leq m\leq k$.

\begin{lem}\label{lem:param_orbit}
Let $c$ be a kaleidoscopic coloring of $D_n$ and say that $R=\{x_1,\dots,x_k\}\subset\Br(D_n)$ and $R'=\{x^\prime_1,\dots,x^\prime_k\}\subset\Br(D_n)$ are center-closed subsets such that there is a partial dendrite morphism $f\colon R\rightarrow R'$ with $f(x_i)=x_i^\prime$. There is then $g\in\Uf(\Gamma)$ such that $g(x_i)=x_i^\prime$ for all $i$ if and only if for all $i$, there is $\gamma_i\in \Gamma$ such that $\gamma_i (c_{x_i}(x_j))=c_{x_i^\prime}(x_j^\prime)$ for all $j\neq i$.
\end{lem}

\begin{proof}
The condition of the existence of the elements $\gamma_i$ is clearly necessary. Let us prove it also sufficient.

Let $\Omega_R$ be the components determined by $R$ and $\Omega_{R'}$ be the components determined by $R'$. Recall that $(\comp{x}_i)_R$ is the collection of components of $D_n\setminus\{x_i\}$ disjoint from $R$. For each $x_i\in R$ and $A\in (\comp{x}_i)_R$, let $B\in (\comp{x'}_i)_{R'}$ be such that $\gamma_i(c(A))=c(B)$; such an element exists by the properties of $\gamma_i$.  Corollary~\ref{cor:component_maps_1} gives a homeomorphism $h_A\colon A\cup R\rightarrow B\cup R'$ extending $f$ such that $c_y(z)=c_{h_A(y)}(h_A(z))$ for all $y,z\in A$ distinct. For each $x_i,x_j\in R$ adjacent, we apply Corollary~\ref{cor:component_maps_2} to obtain a homeomorphism
$$g_{i,j}\colon C_{x_i,x_j}\cup R\longrightarrow C_{x_i',x_j'}\cup R'$$
extending $f$ such that $c_y(z)=c_{g_{i,j}(y)}(g_{i,j}(z))$ for all $y,z\in C_{x_i,x_j}$ distinct. One verifies that the collection of homeomorphisms $h_A$ and $g_{i,j}$ as $A$ ranges over $A\in (\comp{x_i})_R$ for $1\leq i\leq k$ and $i,j$ range over all $1\leq i,j\leq k$ such that $x_i$ and $x_j$ are adjacent satisfy the  hypotheses of Lemma~\ref{lem:gluing_component_maps}. We thus obtain $g\in \Homeo(D_n)$ which extends all maps $h_A$ and $g_{i,j}$ as well as $f$. Furthermore,
\[
\sigma_c(g,v)=
\begin{cases}
\gamma_i & \text{if }v=x_i,\\
1& \text{otherwise.}
\end{cases}
\]
Therefore, $g\in\Uf(\Gamma)$, and $g$ is the desired element.
\end{proof}

An action $G\acts X$ is \textbf{oligomorphic} if for any $k\in\Nb$, the diagonal action $G\acts X^k$ has finitely many orbits. It was proved in \cite[Proposition~12.1]{DM_dendrites} that the action $\Homeo(D_n)$ of $\Br(D_n)$ is oligomorphic. As explained therein, this follows essentially from the statement given above as Lemma~\ref{lem:gluing_component_maps}. Relying on our description of orbits of the action $\Uf(\Gamma)\acts \Br(D_n)^k$, we characterize when the kaleidoscopic groups act oligomorphically. 

\begin{thm}\label{thm:oligo}
For any $\Gamma\leq \Sym(\Nb)$ closed and $n\in \Nbee$, $\Uf(\Gamma)$ acts oligomorphically on $\Br(D_n)$ if and only of $\Gamma\acts [n]$ oligomorphically. In particular, $\Uf(\Gamma)$ is oligomorphic for all $\Gamma\leq \Sym(n)$ with $n$ finite.
\end{thm}

\begin{proof}
Assume first that $\Gamma\acts[n]$ is oligomorphic. The action of $\Homeo(D_n)$ on $\Br(D_n)$ is oligomorphic, so it suffices to show that any $\Homeo(D_n)$-orbit in $\Br(D_n)^k$ splits into finitely many $\Uf(\Gamma)$-orbits for any $k\geq 1$. Since any finite set of branch points can be embedded into some finite center-closed set of branch points, it suffices to prove that a $\Homeo(D_n)$-orbit of any center-closed element in $\Br(D_n)^k$ splits into finitely many $\Uf(\Gamma)$-orbits. 

In view of Lemma~\ref{lem:param_orbit}, the $\Uf(\Gamma)$-orbit (inside its $\Homeo(D_n)$-orbit) is parametrized by the orbits of $\Gamma\acts[n]^{\ell_i}$, where $\ell_i$ is the number of elements of $\comp{x_i}$ that meet $\{x_1,\dots,x_k\}$. Since $\Gamma\acts[n]$ is oligomorphic, there are finitely many such orbits and thus the $\Homeo(D_n)$-orbit of $(x_1,\dots,x_k)$ splits into finitely $\Uf(\Gamma)$-orbits. We conclude that $\Uf(\Gamma)$ acts oligomorphically on $\Br(D_n)$.

Conversely, fix $x_0\in\Br(D_n)$ and choose $y_i\in\Br(D_n)$ such that $c_{x_0}(y_i)=i$ for all $i\in[n]$. Since $\Uf(\Gamma)$ acts oligomorphically on $\Br(D_n)$, there are only finitely many orbits under the (partial) action of $\Uf(\Gamma)$ on 
\[
\{(x_0,y_{i_1},\dots,y_{i_k}): i_j\in [n]\}.
\]
 Furthermore, 
 \[
g(x_0,y_{i_1},\dots,y_{i_k})=(x_0,y_{j_1},\dots,y_{j_k})
\]
implies that $j_l=\sigma(g,x_0)(i_l)$. The group $\Gamma$ thus acts on $[n]^k$ with finitely many orbits. We conclude that $\Gamma\acts[n]$ oligomorphically.
\end{proof}

\subsection{Primitivity}
Recall that an action on a set $X$ is called \textbf{primitive} if it is transitive and does not preserve any non-trivial equivalence relation on $X$, the trivial relations being $X\times X$ and the diagonal $\Delta\subseteq X\times X$ (transitivity is in fact automatic unless $X$ consists of two points, which will never occur in our setting). Thus, an action (with $|X|\neq 2$) \emph{fails} to be primitive exactly when there is an invariant partition of $X$ into at least two blocks of at least two points; such a partition is called a \textbf{system of imprimitivity}. 

Recall further that primitivity is equivalent to being transitive with maximal point stabilizers.

Next, recall that the action $G\acts X$ is \textbf{doubly transitive} if the stabilizer $G_{x}$ acts transitively on $X\setminus\{x\}$ for any point $x$. In the same way, the action $G\acts X$ is \textbf{doubly primitive} if $G_{x}$ acts primitively on $X\setminus\{x\}$.

\begin{thm}\label{thm:U(F)_primitive}
For any $n\in \Nbe_{\geq 3}$ and any $\Gamma\leq \Sym(n)$, the group $\Uf(\Gamma)$ acts primitively on $\Br(D_n)$.
\end{thm}

\begin{proof}
It suffices to prove the theorem for $\Gamma$ trivial since $\Uf(1)<\Uf(\Gamma)$. Let thus $G< \Uf(1)$ be a subgroup strictly containing $\Uf(1)_{x}$ for some branch point $x$; we need to prove $G= \Uf(1)$.

There is $g\in G$ such that $y=g(x)\neq x$. Now $\Uf(1)_{y} = g\Uf(1)_{x}g^{-1}$ is also a subgroup of $G$. It thus suffices to show that any $ h\in\Uf(1)$ is a product of elements of $\Uf(1)_{x}
\cup \Uf(1)_{y}\cup \{g\}$.

Let $z= h(x)$; we can assume $z\neq x$. It is either the case that $U_x(z)=U_x(y)$ or $U_y(x)=U_y(z)$. As the proofs are the same, we assume the former holds. Choose $w_1\in\Br(D_n)\cap (x,y)$ such that $c_z(x)=c_{w_1}(x)$; this is possible because $c$ is kaleidoscopic. By Lemma~\ref{lem:param_orbit}, there is $ h_1\in \Uf(1)_{x}$ such that $ h_1(z)=w_1$; this is where the assumption that $U_x(z)=U_x(y)$ is applied. We next  take $w_2\in \Br(D_n)\cap (x,y)$ such that $c_{w_2}(y)=c_{w_1}(y)$ but $c_{w_2}(x)\neq c_x(y)$. Applying again Lemma~\ref{lem:param_orbit}, there is $ h_2\in \Uf(1)_{y}$ such that $ h_2(w_1)=w_2$. 
We then choose $w_3\in\Br(D_n)\cap (x,y)$ such that $c_{w_3}(y)=c_{x}(y)$ and $c_{w_3}(x)=c_{w_2}(x)$. This is possible, because $c_{w_2}(x)\neq c_x(y)$. As before, we obtain $ h_3\in \Uf(1)_{x}$ such that $ h_3(w_2)=w_3$.  The branch point $w_3$ is now such that $c_{w_3}(y)=c_{x}(y)$. There is thus $ h_4\in \Uf(1)_{y}$ such that $ h_4(w_3)=x$. Below is a possible configuration. The $c_j$ are the colors of the components which are depicted by open alcoves, and $c_3\neq c_0$.
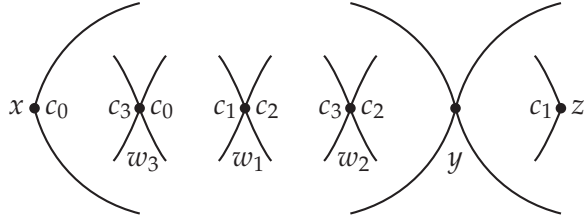
\begin{center}
\setlength{\unitlength}{0.7cm}
\thicklines
\begin{picture}(15,4)

\cbezier(2,2)(2.25,3)(3,3.75)(4,4)
\cbezier(2,2)(2.25,1)(3,.25)(4,0)

\cbezier(10,2)(10.25,3)(11,3.75)(12,4)
\cbezier(10,2)(10.25,1)(11,.25)(12,0)
\cbezier(10,2)(9.75,3)(9,3.75)(8,4)
\cbezier(10,2)(9.75,1)(9,.25)(8,0)
\cbezier(4,2)(4.075,2.25)(4.375,2.875)(4.5,3)
\cbezier(4,2)(4.075,1.75)(4.375,1.125)(4.5,1)
\cbezier(4,2)(3.925,2.25)(3.625,2.875)(3.5,3)
\cbezier(4,2)(3.925,1.75)(3.625,1.125)(3.5,1)
\cbezier(6,2)(6.075,2.25)(6.375,2.875)(6.5,3)
\cbezier(6,2)(6.075,1.75)(6.375,1.125)(6.5,1)
\cbezier(6,2)(5.925,2.25)(5.625,2.875)(5.5,3)
\cbezier(6,2)(5.925,1.75)(5.625,1.125)(5.5,1)
\cbezier(8,2)(8.075,2.25)(8.375,2.875)(8.5,3)
\cbezier(8,2)(8.075,1.75)(8.375,1.125)(8.5,1)
\cbezier(8,2)(7.925,2.25)(7.625,2.875)(7.5,3)
\cbezier(8,2)(7.925,1.75)(7.625,1.125)(7.5,1)
\cbezier(12,2)(11.925,2.25)(11.625,2.875)(11.5,3)
\cbezier(12,2)(11.925,1.75)(11.625,1.125)(11.5,1)

\put(2,2){\circle*{0.2}}\put(1.5,1.9){$x$}
\put(10,2){\circle*{0.2}}\put(9.85,.9){$y$}
\put(12,2){\circle*{0.2}}\put(12.2,1.9){$z$}
\put(4,2){\circle*{0.2}}\put(3.75,.9){$w_3$}
\put(6,2){\circle*{0.2}}\put(5.75,.9){$w_1$}
\put(8,2){\circle*{0.2}}\put(7.75,.9){$w_2$}
\put(3.4,1.9){$c_3$}
\put(5.4,1.9){$c_1$}
\put(7.4,1.9){$c_3$}
\put(11.4,1.9){$c_1$}
\put(4.2,1.9){$c_0$}
\put(6.2,1.9){$c_2$}
\put(8.2,1.9){$c_2$}
\put(2.2,1.9){$c_0$}
\end{picture}
\end{center}
We deduce that 
\[
 h_4\circ h_3\circ h_2\circ h_1\circ  h(x)=x.
\]
with $ h_i\in \Uf(1)_{x}\cup \Uf(1)_{y}$ for each $i$. It follows that $ h\in G$ as claimed.
\end{proof}

Since the actions of $\Uf(\Gamma)$ on $\Ends(D_n)$ and $\Reg(D_n)$ are doubly transitive, by Proposition~\ref{prop:U(1)_trans}, Theorem~\ref{thm:U(F)_primitive} completes the proof of the following corollary.
\begin{cor}
For $n\in \Nbee$ and $\Gamma\leq \Sym(n)$, $\Uf(\Gamma)$ acts primitively on each of $\Br(D_n)$, $\mathrm{Ends}(D_n)$ and $\mathrm{Reg}(D_n)$. 
\end{cor}

\subsection{Fixed points and sets}
We conclude this section by noting several easy lemmas for later use.

\begin{lem}\label{lem:fixed_pt_set}
Say that $n\geq \Nbee$, $F\subseteq \Br(D_n)$ is finite and center-closed and let $\Gamma\leq \Sym(n)$. Setting $G=\Uf(\Gamma)$, every $x\in D_n\setminus F$ has an infinite orbit under $G_{(F)}$. In particular, the fixed point set of $G_{(F)}$ is exactly $F$.
\end{lem}
\begin{proof}
Fix $c$ a kaleidoscopic coloring of $D_n$ and observe that it suffices to prove the lemma for $\Gamma=1$. 

Fix $x\in D_n\setminus F$. As the arguments are similar, we consider the case that $x$ is a branch point. Let $\Omega_F$ be the components determined by $F$ and say that $x$ is in some component $A\in \Omega_F$.  The component $A$ is either an element of $\comp{f}_F$ for some $f\in F$ or $A=C_{x,y}$ for some adjacent $x,y\in F$.  
As the cases are similar, let us suppose that $A=C_{x,y}$ and fix $a\in A$. Since $c$ is kaleidoscopic, there are infinitely many $a'\in A\cap \Br(D_n)$ such that $c_a'(x)=c_a(x)$ and $c_{a'}(y)=c_a(y)$. For any such $a'$, Lemma~\ref{lem:param_orbit} supplies $g\in G_{(F)}$ such that $g(a)=a'$. The orbit of $a$ under $G_{(F)}$ is thus infinite.
%
%
%
\end{proof}

\begin{lem}\label{lem:fixed_interval}
Say that $n\geq \Nbee$, $x,y\in D_n$ are distinct and $\Gamma\leq \Sym(n)$. Setting $G=\Uf(\Gamma)$, the only arc invariant under $G_{x}\cap G_{y}$ is $[x,y]$.
\end{lem}

\begin{proof}
Suppose $[z,z']$ is an arc invariant under the action of $G_{x}\cap G_{y}$. The element $z$ thus has a finite orbit under the action of $G_{x}\cap G_{y}$. In view of Lemma~\ref{lem:fixed_pt_set}, we deduce that $\{z,z'\}=\{x,y\}$. 
\end{proof}

\begin{lem}\label{lem:fixed_pt_set_tripod}
Say that $n\geq \Nbee$, $x,y,z\in \Br(D_n)$ are distinct and $\Gamma\leq \Sym(n)$. Setting $G=\Uf(\Gamma)$, the fixed point set of $G_{(x,y,z)}$ is exactly $\{x,y,z,\kappa(x,y,z)\}$.
\end{lem}

\begin{proof}
Any $g\in\Homeo(D_n)$ fixing $x,y$ and $z$ also fixes their center $\kappa(x,y,z)$. Since the finite set $\{x,y,z,\kappa(x,y,z)\}$ is center-closed, the lemma is now a consequence of Lemma~\ref{lem:fixed_pt_set}.
\end{proof}

\section{Simplicity and uniform perfectness}
The aim of this section is to prove the simplicity of all $\Uf(\Gamma)$. We stress that there is no assumption at all on $\Gamma$. 

For the dendrite $D_n$, a coloring $c$ and $x\in \Br(D_n)$, recall that $U_{x,i}$ denotes the element of $\comp{x}$ with color $i$.
\begin{lem}\label{lem:austro-boreal-coloring}
Let $n\in\Nbee$, $\gamma\in \Sym(n)$ and $y,z\in\Br(D_n)$ be distinct. For any $x_1,x_2\in (y,z)$ such that $\gamma (c_{x_1}(y))=c_{x_2}(y)$ and $\gamma (c_{x_1}(z))=c_{x_2}(z)$ and for any family of homeomorphisms $f_i\colon U_{x_1,i}\to U_{x_2,\gamma (i)}$ for $i\in [n]\setminus\{c_{x_1}(y),c_{x_1}(z)\}$, there is $h\in\Homeo(D_n)$ such that
\begin{itemize}
\item  $h$ is trivial on $D_n\setminus C_{x_1,x_2}$,
\item $h(x_1)=x_2$,
\item $\sigma(h,x_1)=\gamma$,
\item $h\rest_{U_{x_1,i}}=f_i$ for any $i\in[n]\setminus\{c_{x_1}(y),c_{x_1}(z)\}$, 
\item and $\sigma(h,x)=1$ for any $x\notin \cup_{i\in[n]\setminus\{c_{x_1}(y),c_{x_1}(x)\}}U_{x_1,i}\cup\{x_1\}$.
\end{itemize}
\end{lem}

\begin{proof}
Let $R=\{y,z,x_1\}$ and $S=\{y,z,x_2\}$ and let $f:R\rightarrow S$ by $y\mapsto y$, $z\mapsto z$ and $x_1\mapsto x_2$. This map is a partial dendrite map. For each $U\in \comp{y}_R\cup \comp{z}_R$, we define $k_U:U\cup R\rightarrow U\cup S$ to be the identity map. For $C_{y,x_1}$ and $C_{x_1,z}$, Corollary~\ref{cor:component_maps_2} supplies homeomorphisms $h_y:C_{y,x_1}\cup R\rightarrow C_{y,x_2}\cup S$ and $h_z:C_{x_1,z}\cup \rightarrow C_{x_2,z}$ that preserve the coloring at all branch points $w\in C_{y,x_1}$ and $v\in C_{x_1,z}$, respectively. Extending each $f_i$ so that they extend $f$, the collection of maps 
\[
\{h_y,h_z\}\cup\{k_U\}\cup\{f_i\}
\]
meet the hypotheses of Lemma~\ref{lem:gluing_component_maps}. We thereby obtain a homeomorphism $h$ of $D_n$, and one verifies that $h$ has the desired properties.
\end{proof}

\begin{prop}\label{gcomp} For $n\in\Nbee$ and $\Gamma\leq \Sym(n)$, $\Uf(\Gamma)$ is generated by pointwise stabilizers of components. More precisely, for every $g\in \Uf(\Gamma)$ there are $g_1$, $g_2$ and $g_3$ in $\Uf(\Gamma)$ such that $g=g_1g_2g_3$ and each $g_i$ pointwise stabilizes some component $A_i\in \comp{D}_n$. 
\end{prop}

\begin{proof}
Fix $g\in\Uf(\Gamma)$. We aim to write $g$ as a finite product of elements in $\Uf(\Gamma)$ each fixing pointwise some component. Let $x$ be a branch point of $D_n$ that is not fixed by $g$. 

\medskip

\textbf{First case:} $U_{g(x)}(g^{2}(x))=g(U_x(g(x)))\neq U_{g(x)}(x)$; recall that $U_y(z)$ is the element of $\comp{y}$ containing $z$. The following diagram illustrates this case. The open alcoves depict the connected components.
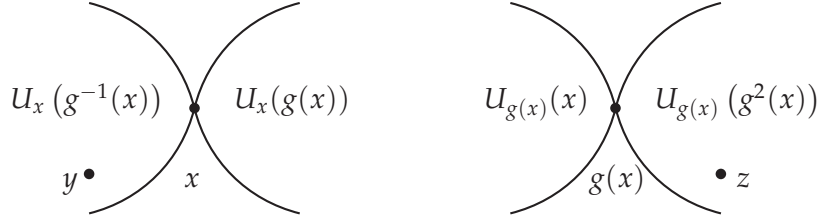
\begin{center}
\setlength{\unitlength}{0.7cm}
\thicklines
\begin{picture}(12,4)
\cbezier(2,2)(2.25,3)(3,3.75)(4,4)
\cbezier(2,2)(2.25,1)(3,.25)(4,0)
\cbezier(2,2)(1.75,3)(1,3.75)(0,4)
\cbezier(2,2)(1.75,1)(1,.25)(0,0)
\cbezier(10,2)(10.25,3)(11,3.75)(12,4)
\cbezier(10,2)(10.25,1)(11,.25)(12,0)
\cbezier(10,2)(9.75,3)(9,3.75)(8,4)
\cbezier(10,2)(9.75,1)(9,.25)(8,0)
\put(2,2){\circle*{0.2}}\put(1.8,.5){$x$}
\put(10,2){\circle*{0.2}}\put(9.5,.5){$g(x)$}
\put(0,.75){\circle*{0.2}}\put(-.5,.5){$y$}
\put(12,.75){\circle*{0.2}}\put(12.3,.5){$z$}
\put(2.75,2){$U_x(g(x))$}
\put(10.75,2){$U_{g(x)}\left(g^2(x)\right)$}
\put(7.5,2){$U_{g(x)}(x)$}
\put(-1.5,2){$U_{x}\left(g^{-1}(x)\right)$}
\end{picture}
\end{center}
Set $i=c(U_x(g(x)))$ and $j=c(U_{x}(g^{-1}(x)))$; note that $i\neq j$. Fix branch points $y\in U_{x}(g^{-1}(x))$ and $z\in U_{g(x)}(g^2(x))$. The element $g$ sends $U_{x}(g^{-1}(x))$ to $U_{g(x)}(x)$ and $U_x(g(x))$ to $U_{g(x)}(g^{2}(x))$. Therefore, $\sigma(g,x)(c_x(y))=c_{g(x)}(y)$, and $\sigma(g,x)(c_x(z))=c_{g(x)}(z)$. The branch points $x,g(x)$ are elements of $(z,y)$, so we are in a position to apply Lemma~\ref{lem:austro-boreal-coloring} with $\gamma=\sigma(g,x)$ and $f_k:=g\rest_{U_{x,k}}$ where $k\in [n]\setminus \{i,j\}$. We thereby obtain $h\in\Homeo(D_n)$ such that 

\begin{itemize}
\item $h$ is trivial outside $C_{y,z}$,
\item $h(x)=g(x)$ and $\sigma(h,x)=\sigma(g,x)$,
\item $h\rest_{U_{x,k}}=g\rest_{U_{x,k}}$ for $k\in [n]\setminus\{ i,j\}$,  and
\item $\sigma(h,v)=1$ for any $v\notin \cup_{k\in[n]\setminus\{i,j\}}U_{x,k}\cup\{x\}$.
\end{itemize}

By construction, all $\sigma(h,v)$ are trivial or coincide with some $\sigma(g,v')^{-1}$, and thus they belong to $\Gamma$. Hence, $h\in\Uf(\Gamma)$ and fixes pointwise a component. Moreover, $l=h^{-1}g$ fixes pointwise the component $U_{x,k}$ for any $k\neq i,j$. As $g=hl $, the lemma is proved in this case.

\medskip

\textbf{Second case:} $g(U_x(g(x))= U_{g(x)}(x)$. Let $V\in \comp{x}\setminus\{U_x(g(x))\}$ and choose $y\in V$. Let $W\in \comp{g(x)}\setminus\{g(V), U_{g(x)}(x)\}$ and choose $z\in W$. The following diagram illustrates this case.

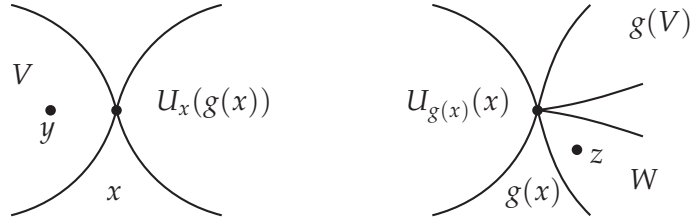
\begin{center}
\setlength{\unitlength}{0.7cm}
\thicklines
\begin{picture}(12,4)
\cbezier(2,2)(2.25,3)(3,3.75)(4,4)
\cbezier(2,2)(2.25,1)(3,.25)(4,0)
\cbezier(2,2)(1.75,3)(1,3.75)(0,4)
\cbezier(2,2)(1.75,1)(1,.25)(0,0)
\cbezier(10,2)(10.25,3)(10.5,3.5)(11,4)
\cbezier(10,2)(10.25,1)(10.5,.5)(11,0)
\cbezier(10,2)(10.75,2.1)(11.25,2.25)(12,2.5)
\cbezier(10,2)(10.75,1.9)(11.25,1.75)(12,1.5)
\cbezier(10,2)(9.75,3)(9,3.75)(8,4)
\cbezier(10,2)(9.75,1)(9,.25)(8,0)
\put(2,2){\circle*{0.2}}\put(1.8,.3){$x$}
\put(10,2){\circle*{0.2}}\put(9.4,.3){$g(x)$}
\put(.75,2){\circle*{0.2}}\put(.55,1.5){$y$}
\put(10.75,1.25){\circle*{0.2}}\put(11,1){$z$}
\put(2.75,2){$U_x(g(x))$}
\put(7.5,2){$U_{g(x)}(x)$}
\put(0,2.5){$V$}
\put(11.75,3.5){$g(V)$}
\put(11.75,.5){$W$}
\end{picture}
\end{center}

As the coloring is kaleidoscopic, we may find $w\in (y,x)$ such that $c_w(z)=c_{g(x)}(z)$ and $c_w(y)=c_{g(x)}(y)$.  It then follows from Lemma~\ref{lem:austro-boreal-coloring} then there is $h\in\Uf(1)$ such that 

\begin{itemize}
\item $hg(x)=w$ and
\item $h$ is trivial outside $C_{y,z}$.
\end{itemize}
We deduce further that $hg(U_x(hg(x))\neq U_{hg(x)}(x)$. Indeed,
\[
U_{hg(x)}(x)=U_{hg(x)}(z)=h(U_{g(x)}(z))=h(W)\neq hg(V),
\]
but
\[
hg(U_x(hg(x)))=hg(U_x(y))=hg(V).
\]

The element $hg$ is in $\Uf(\Gamma)$ and satisfies the condition of the first case, so we may write $hg=g_1g_2$ such that $g_1$ and $g_2$ are elements of $\Uf(\Gamma)$  that fix a component. Hence, $g=h^{-1}g_1g_2$, and as $h$ also fixes a component, the lemma is verified in this case.
\end{proof}

An arc $[x,y]\subseteq D_n$ is called \textbf{austro-boreal} for a homeomorphism $g\in \Homeo(D_n)$ if $\mathrm{Fix}(g)\cap [x,y]=\{x,y\}$, where $\mathrm{Fix}(g)$ is the collection of fixed points of $g$ in $D_n$. A subgroup $H\leq\Homeo(D_n)$ is called \textbf{dendro-minimal} if the smallest $H$-invariant sub-dendrite of $D_n$ is $D_n$. 

\begin{thm}\label{thm:simple}
For $n\in\Nbee$ and $\Gamma\leq \Sym(n)$, $\Uf(\Gamma)$ is simple as an abstract group.
\end{thm}

\begin{proof}
Let $N$ be a non-trivial normal subgroup of $\Uf(\Gamma)$. The action of $\Uf(\Gamma)$ on $\Br(D_n)$ is transitive, so $\Uf(\Gamma)$ is dendro-minimal. By~\cite[Lemma 4.3]{DM_structure}, $N$ is also dendro-minimal.

In view of Proposition~\ref{gcomp}, it suffices to show that every $g\in \Uf(\Gamma)$ fixing pointwise a component  $Y$ of $D_n\setminus\{x\}$ for some $x\in D_n$ belongs to $N$. Fix such a $g\in \Uf(\Gamma)$ and let $Y\in \comp{x}$ be the component fixed by $g$.

From \cite[Theorem~10.5]{DM_dendrites}, $N$ contains an element $n$ admitting an austro-boreal arc $I=[y,z]$. Lemma~\ref{lem:param_orbit} ensures that we can assume, upon conjugating $n$, that $I$ lies in $Y$ and that the image $b$ of $x$ under the first-point map to $I$ is some branch point in the interior of $I$.

The action of $\grp{n}$ on $I\setminus \{y,z\}$ is free by definition of austro-boreal arcs. For each $t\in I$, we may then define
\[
h_t=
\begin{cases}
n^k g n^{-k} & t=n^kb \text{ for some }k\geq 0\\
1 & \text{otherwise.}
\end{cases}
\] 
For each $t\in I$, let 
\[
X_t=\bigcup_{A\in \hat{t}\setminus\{U_t(y),U_t(z)\}}A;
\]
the set $X_t$ is exactly the collection of $v\in D_n\setminus\{t\}$ such that $r(v)=t$ where $r$ is the first point map onto $I$. The function $h_t$ is a homeomorphism of $D_n$ trivial outside $X_t$. Appealing to Lemma~\ref{lem:patchwork2}, there is $h\in \Homeo(D_n)$ such that $h$ fixes $I$ and $h\rest_{X_t}=h_t$ for all $t\in I$.  One checks further that for any $v\in\Br(D_n)$, $\sigma(h,v)$ is trivial or coincides with $\sigma(n^k g n^{-k},v)$, and thus $h\in\Uf(\Gamma)$. 

An easy computation shows that $[h,n]$ acts trivially except on $X_b$, and on $X_b$, it acts like $g$. We thus deduce that $g=[h,n]$ and so $g\in N$.
\end{proof}

A simple group is in particular perfect; thus, each of its elements is a product of commutators. This does not mean, of course, that every element is a commutator. A group is called \textbf{uniformly perfect} if there is an integer $k$ such that every element is a product of $k$ commutators. The smallest such integer $k$ for a given element is called its \textbf{commutator length}.

From the proof of Theorem~\ref{thm:simple}, we obtain the following.

\begin{thm}\label{thm:uniform_perfect}
The group $\Uf(\Gamma)$ is uniformly perfect. More precisely, every element is the product of three commutators.
\end{thm}

\begin{proof} In the proof of Theorem~\ref{thm:simple}, we may take $N$ to be $\Uf(\Gamma)$ itself, so any $g\in \Uf(\Gamma)$ fixing pointwise a component is a commutator. Thanks to Proposition~\ref{gcomp}, any element of $\Uf(\Gamma)$ is the product of at most three elements each pointwise fixing some component. Hence, every element of $\Uf(\Gamma)$ is the product of three commutators.
\end{proof}

\section{Isomorphism types}
This section investigates to what extent the isomorphism type of $\Uf(\Gamma)$ depends on $\Gamma$; in particular, we obtain a continuum of non-isomorphic kaleidoscopic groups. Along the way, we show that many kaleidoscopic groups have a unique Polish topology.

\subsection{Unique Polish topology}
Throughout this subsection, we fix some $n\in \Nbee$. Consider a closed subgroup $G< \Homeo(D_n)$, a branch point $x\in \Br(D_n)$ and a component $U\in \comp x$. The \textbf{rigid stabilizer} of $U$ in $G$ is defined to be
\[
\rist_G(U)=\{g\in G: g(y)=y\text{ for all }y\notin U\}.
\]
The \textbf{rigid stabilizer} of $x$ in $G$ is defined to be 
\[
\rist_G(x)=\ol{\grp{\rist_{G}(V): V\in \comp x}}.
\]
By Proposition~\ref{prop:closed subspace top}, we may  consider rigid stabilizers for $\Uf(\Gamma)$ with $\Gamma$ closed.

\begin{lem}\label{lem:rist:patch}
For $\Gamma\leq \Sym(n)$ closed and $G=\Uf(\Gamma)$, the rigid stabilizer $\rist_G(x)$ is isomorphic to the direct product of the rigid stabilizers of the components $V\in\comp x$.
\end{lem}

\begin{proof}
The restriction to components yields an injective continuous homomorphism 
$$\rist_G(x) \longrightarrow \prod_{V\in \comp x}  \rist_G(V).$$
It remains to show the homomorphism is also surjective. 

Let $V_i$ be an enumeration of the components at $x$ and consider any sequence $h_i \in \rist_G(V_i)$. The elements $h_i$ patch together to form a homeomorphism $h$ of $D_n$ via Lemma~\ref{lem:patchwork2}. Given any $j\in\Nb$, denote by $g_j$ the homeomorphism obtained by patching together $h_i$ for $i\leq j$ and the identity on $V_i$ when $i>j$. The element $g_j$ has all its local actions in $\Gamma$, and $g_j\in \rist_G(x)$. Furthermore, $g_j$ converges to $h$ pointwise. Hence, $h\in \rist_G(x)$, and the map in question is surjective.
\end{proof}

We now describe the centralizer $Z_G$ in $G$ of rigid stabilizers.

\begin{lem}\label{lem:Z_G-rist_G}
Let $\Gamma\leq \Sym(n)$ be closed and $G=\Uf(\Gamma)$. For any $x\in \Br(D_n)$ and $V\in \comp x$, 
\[
Z_{G}(\rist_G(V))=G_{(V)}.
\]
\end{lem}

\begin{proof}
Fix $c$ a kaleidoscopic coloring of $D_n$, set $L=Z_G(\rist_G(V))$ and suppose toward a contradiction that some $h\in L$ acts non-trivially on $V$. Say that $v\in V$ is a branch point such that $h(v)\neq v$.

Letting $r$ be the first point map onto $[v,x]$, we have two cases: (1) $r(h(v))\in \{v,x\}$, and (2) $r(h(v))\in (v,x)$. The first case is easier than and similar to the second case, so we shall only address case (2). 

For case (2), set $z=r(h(v))$, take $y\in V$ such that $v\in (y,z)$ and find $w\in (y,z)\setminus\{v\}$ such that $c_w(y)=c_v(y)$ and $c_w(z)=c_v(z)$. For each $i\in [n]\setminus\{c_w(y),c_w(z)\}$, Corollary \ref{cor:component_maps_1} supplies a homeomorphism $f_i:U_{w,i}\rightarrow U_{v,i}$ that preserves the coloring.
We now apply Lemma~\ref{lem:austro-boreal-coloring} for $\gamma=1$, $v,w\in (z,y)$ and the family $(f_i)$. This yields $g\in \Uf(\{1\}\acts [n])$ such that $g$ acts trivially on $D_n\setminus C_{y,z}$ and $g(w)=v$. 

The element $g$ is an element of $\rist_G(V)$ and fixes $h(v)$. On the other hand, $h$ commutes with $g$. Hence, $h(v)=gh(v)=hg(v)=h(w)$. This is absurd since $w\neq v$, and thus, $h$ fixes $V$. We conclude that $Z_G(\rist_G(V))\leq G_{(V)}$. The converse inclusion is immediate.
\end{proof}

\begin{cor}\label{cor:rist-centralizer}
Let $\Gamma\leq \Sym(n)$ be closed and $G=\Uf(\Gamma)$. For any $x\in \Br({D}_n)$ and $V\in \comp x$, 
\[
\pushQED{\qed} \rist_G(V)=\bigcap_{U\in \comp x\setminus\{V\}}Z_G(\rist_G(U)).\qedhere\popQED
\]
\end{cor}

We can now establish the relationship between the stabilizer and the rigid stabilizer of a branch point.

\begin{lem}\label{lem:pt_stab_factorization}
Let $\Gamma\leq \Sym(n)$ be closed and $G=\Uf(\Gamma)$. For any $x\in \Br({D}_n)$, there is a closed subgroup $\Gamma_x\leq G_{x}$ such that $\Gamma_{x}\rightarrow \Gamma$ by $g\mapsto \sigma(g,x)$ is an isomorphism, $\Gamma_x$ normalizes $\rist_G(x)$, and $G_{x}= \rist_G(x)\Gamma_x$. 
\end{lem}

In view of Lemma~\ref{lem:rist:patch}, we conclude:

\begin{cor}\label{cor:wreath}
The stabilizer $G_{x}$ is isomorphic to the permutational wreath product
\[\pushQED{\qed}\left(\prod_{V\in \comp x}  \rist_G(V)\right) \rtimes \Gamma. \qedhere\popQED\]
\end{cor}

\begin{proof}[Proof of Lemma~\ref{lem:pt_stab_factorization}]
The group $\Gamma_x$ is provided by Proposition~\ref{prop:local:split}, which moreover ensures that the canonical morphism $\Gamma_x\to\Gamma$ obtained from Lemma~\ref{lem:continuous_cocycle} is an isomorphism of topological groups.

That $\Gamma_x$ normalizes $\rist_G(x)$ is immediate from the construction of $\Gamma_x$ given in the proof of Proposition~\ref{prop:local:split}. It thus remains only to show that every element $h\in G_{x}$ lies in $\rist_G(x)\Gamma_x$.

Upon multiplying by an element of $\Gamma_x$, we can assume that $h$ fixes each $V\in \comp x$ setwise. By restricting to each $V$, we thus obtain an element in the product $\prod_{V\in \comp x}  \rist_G(V)$. The proof of  Lemma~\ref{lem:rist:patch} shows that $h$ belong to $\rist_G(x)$.
\end{proof}

It now follows that $\Uf(\Gamma)$ has a unique Polish topology for $\Gamma$ discrete.

\begin{thm}\label{thm:unique_polish_top}
If $\Gamma\leq \Sym(n)$ is discrete, then $\Uf(\Gamma)$ has a unique Polish group topology. 
\end{thm}
\begin{proof}
Suppose $\tau$ is Polish topology on $G=\Uf(\Gamma)$ and fix $x\in \Br({D}_n)$. For $U\in \comp x$, Corollary~\ref{cor:rist-centralizer} implies that $\rist_G(U)$ is an intersection of centralizers. As centralizers are always closed, $\rist_G(U)$ is closed in the $\tau$-topology. The subgroup $Z_G(\rist_G(U))$ is also closed in the $\tau$-topology, so  $L=\rist_G(U)Z_G(\rist_G(U))$ is an analytic set. Indeed, $L$ is the image of the $\tau$-closed set $\rist_G(U)\times Z_G(\rist_G(U))$ under the multiplication map, which is continuous. Furthermore, $\rist_G(x)\leq L$, so by Lemma~\ref{lem:pt_stab_factorization}, $L$ has countable index. Recalling that analytic sets are measurable in the sense of Baire~\cite[Theorem~21.6]{Kechris95}, it follows that $L$ is open in the $\tau$-topology~\cite[Theorem~9.9]{Kechris95}. Hence, $G_{x}$ is open in the $\tau$ topology, and $\tau$ refines the usual topology on $\Uf(\Gamma)$. On the other hand, a Polish group does not admit any properly refining Polish group topology, because every continuous and bijective homomorphism between Polish groups is an isomorphism of topological groups; see e.g.~\cite[Theorem~2.1]{Effros}. We conclude that $\tau$ is in fact equal to the usual group topology.
\end{proof}

\subsection{Isomorphic groups}

Our next few lemmas consider setwise invariant arcs. We stress that an element setwise stabilizing an arc can reverse the orientation.
\begin{lem}\label{lem:normalizer_arc}
Let $n\in \Nbee$ and $G\leq \Homeo(D_n)$. Suppose that $G$ fixes an arc $[x,y]$ setwise  and fixes setwise no proper sub-arc $[x',y']\subseteq [x,y]$. If $h\in \Homeo(D_n)$ normalizes $G$, then $h$ fixes $[x,y]$ setwise.
\end{lem}
\begin{proof}
The arc $[h(x),h(y)]$ is also invariant under the action of $G$. Letting $r$ be the first point map onto $[x,y]$, the arc $[r(h(x)),r(h(y))]$ is a sub-arc or point of $[x,y]$ that is invariant under $G$, and as no proper such arc or point exists, we may assume, without loss of generality, that $r(h(x))=x$ and $r(h(y))=y$.

The geodesic $[h(x),x]$  does not contain $h(y)$ since $x=r(h(x))\neq r(h(y))=y$. We conclude that 
\[
[h(x),x]\cup[x,y]\cup [y,h(y)]=[h(x),h(y)].
\]
On the other hand, $[h(x),h(y)]$ contains no proper sub-arc invariant under the action of $G$, since $[x,y]$ does not, hence $[h(x),h(y)]=[x,y]$. The arc $[x,y]$ is thus fixed setwise by $h$.
\end{proof}

\begin{lem}\label{lem:G_(v)_image_fixed_pt}
For $m,n\in \Nbee$, suppose that $\Delta\leq \Sym(n)$ and $\Gamma\leq \Sym(m)$ are discrete groups. Set $G=\Uf(\Delta)$ and $H=\Uf(\Gamma)$ and suppose that $\fhi:G\rightarrow H$ is an isomorphism of Polish groups. For any $v\in \Br({D}_n)$, one of the following hold:
\begin{enumerate}
\item There is an arc $[x,y]$ such that $\fhi(G_{v})$ fixes $[x,y]$ setwise and fixes setwise no proper sub-arc of $[x,y]$, or
\item $\fhi(G_{v})=H_{w}$ for some $w\in \Br({D}_m)$.
\end{enumerate}
\end{lem}

\begin{proof}
Let $\Delta_v\leq G_{v}$ be as given by Lemma~\ref{lem:pt_stab_factorization}. Let us write $\rist_G(v)=\prod_{i\in [n]}L_i$ where $L_i=\rist_G(U_{v,i})$ and $U_{v,i}$ is the element of $\comp{v}$ with color $i$.  By \cite[Corollary 4.6]{DM_dendrites}, some $\fhi(L_i)$ fixes a point or a pair of points.

Suppose first that some $\fhi(L_i)$ fixes a point; without loss of generality, we assume that $\fhi(L_0)$ fixes a point. Let $X_0\subseteq D_m$ be the fixed point set of $\fhi(L_0)$. For $j\neq 0$, we may find $y\in \Br(D_n)\cap {U}_{v,0}$ such that $U_{y,j}\subseteq U_{v,0}$, since the coloring is kaleidoscopic. The group $\Uf(1)$ acts transitively on $\Br({D}_n)$, so there is $g\in \Uf(1)\leq G$ such that $g(v)=y$. It follows that $g(U_{v,j})=U_{y,j}$, and therefore, $gL_jg^{-1}\leq L_0$. The group $\fhi(L_j)$ thus fixes $\fhi(g^{-1})(X_0)$. We conclude that every $\fhi(L_j)$ fixes some element of $D_m$. Applying {\cite[Lemma 2.11]{DM_dendrites}}, $\fhi(\rist_G(v))$ has a fixed point.

Let $Y$ be the fixed point set of $\fhi(\rist_G(v))$. Lemma~\ref{lem:pt_stab_factorization} ensures that $\fhi(\rist_G(v))$ is of countable index in $H$, so $\fhi(\rist_G(v))$ is open in $H$. There is thus a finite set of branch points $Z$ such that 
\[
H_{(Z)}\leq \fhi(\rist_G(v))\leq H_{(Y)}.
\]
Appealing to Lemma~\ref{lem:fixed_pt_set}, it is the case that $Y=Z$, so $Y$ is finite. The image $\fhi(\Delta_v)$ normalizes $\fhi(\rist_G(v))$, so $\fhi(\Delta_v)$ fixes $Y$ setwise. The group $\fhi(G_{v})=\fhi(\rist_G(v))\fhi(\Delta_v)$ therefore setwise fixes $Y$. Applying {\cite[Proposition 3.2]{DM_dendrites}}, $\fhi(G_{v})$ acts elementarily on $D_m$. If $\fhi(G_{v})$ fixes a point $w$, then $\fhi(G_{v})=H_{w}$ since $G_{v}$ is a maximal subgroup of $G$ via Theorem~\ref{thm:U(F)_primitive}. In view of Corollary~\ref{cor:U(F)_trans}, we deduce further that $w$ is a branch point, since $\fhi(G_{v})$ has countable index in $H$, so claim (2) holds. Otherwise, $\fhi(G_{v})$ setwise stabilizes some arc $[x,y]$. Up to passing to a sub-arc, we may assume that  $\fhi(G_v)$ fixes $[x,y]$ setwise and fixes setwise no proper sub-arc $[x',y'] \subset[x,y]$. Hence, claim (1) holds.

Suppose next that no $\fhi(L_i)$ fixes a point. Without loss of generality, $\fhi(L_0)$ fixes an arc $[x,y]$ setwise, and we may assume further that $\fhi(L_0)$ setwise stabilizes no proper sub-arc of $[x,y]$. In view of Lemma~\ref{lem:normalizer_arc}, $[x,y]$ is in fact invariant under the action of $\fhi(\rist_G(v))$, and $[x,y]$ contains no proper setwise invariant sub-arc. The group $ \fhi(\Delta_v)$ normalizes $\fhi(\rist_G(v))$, so by a second application of Lemma~\ref{lem:normalizer_arc}, $[x,y]$ is invariant under the action of $\fhi(G_{v})$. The arc $[x,y]$ also contains no proper  setwise invariant sub-arc, so claim (2) holds.
\end{proof}

We now eliminate case (1) of the previous lemma.
\begin{lem}\label{lem:pt_stab_iso}
For $m,n\in \Nbee$, suppose that $\Delta\leq \Sym(n)$ and $\Gamma\leq \Sym(m)$ are discrete groups. Set $G=\Uf(\Delta)$ and $H=\Uf(\Gamma)$ and suppose that $\fhi:G\rightarrow H$ is an isomorphism of Polish groups. For any $v\in \Br({D}_n)$, there is $w\in \Br({D}_m)$ such that $\fhi(G_{v})=H_{w}$. 
\end{lem}

\begin{proof}
Via Lemma~\ref{lem:G_(v)_image_fixed_pt}, either $\fhi(G_{v})=H_{w}$ for some $w\in \Br({D}_m)$, or $\fhi(G_{v})\leq H_{\{x,y\}}$ for some $x\neq y$ in $D_m$. Suppose toward a contradiction the latter case holds. Note that since $G$ acts on $\Br(D_n)$ transitively, the latter case holds for all $v\in \Br(D_m)$.

The group $G_{v}$ is maximal in $G$ by Theorem~\ref{thm:U(F)_primitive}, so  $\fhi(G_{v})=H_{\{x,y\}}$. In view of Lemma~\ref{lem:param_orbit}, we may find $h=\fhi(g)$ such that $h(\{x,y\})=\{x',y'\}$ with $[x,y]\cap [x',y']=\emptyset$. Setting $g(v)=:v'$, we infer that $\fhi(G_{v'})=H_{\{x',y'\}}$.

The arcs $[x,y]$ and $[x',y']$ are disjoint, so there is $z\in \Br(D_m)$ such that $H_{\{x,y\}}\cap H_{\{x',y'\}}\leq H_{z}$. Applying Lemma~\ref{lem:G_(v)_image_fixed_pt} to $\fhi^{-1}:H\rightarrow G$, we see that $\fhi^{-1}(H_{z})$ equals $G_{w}$ or is contained in $G_{\{w,w'\}}$. The reductio hypothesis excludes the former case, and we deduce that $\fhi^{-1}(H_{z})=G_{\{w,w'\}}$. 

On the other hand, 
\[
\fhi^{-1}(H_{\{x,y\}}\cap H_{\{x',y'\}})=G_{v}\cap G_{v'}\leq \fhi^{-1}(H_{z})=G_{\{w,w'\}},
\]
so Lemma~\ref{lem:fixed_interval} implies that $\{w,w'\}=\{v,v'\}$. We conclude that $H_{\{x,y\}}\cap H_{\{x',y'\}}$ is of index at most two in $H_{z}$, and this is absurd.  
\end{proof}

\begin{thm}\label{thm:isom}
For $m,n\in \Nbee$, suppose that $\Delta\leq \Sym(n)$ and $\Gamma\leq \Sym(m)$ are discrete groups. Then the following are equivalent.
\begin{enumerate}
\item $(\Delta,[n])\simeq (\Gamma,[m])$ as permutation groups.
\item $\Uf(\Delta)\simeq \Uf(\Gamma)$ as abstract groups.
\item $\Uf(\Delta)\simeq \Uf(\Gamma)$ as Polish groups.
\item There is a homeomorphism $\fhi:D_n\rightarrow D_m$ and kaleidoscopic colorings $c$ and $d$ such that $\fhi\Uf_c(\Delta)\fhi^{-1}=\Uf_d(\Gamma)$.
\end{enumerate}
\end{thm}
\begin{proof}
The equivalence of (2) and (3) is given by Theorem~\ref{thm:unique_polish_top}.

\medskip

For (1) implies (3), suppose that $(\Delta,[n])\simeq (\Gamma,[m])$ as permutation groups, so $n=m$. Say that $f:[n]\rightarrow [n]$ is a bijection giving the isomorphism $(\Delta,[n])\rightarrow (\Gamma,[n])$ as permutation groups. Let $c$ be a kaleidoscopic coloring and form $\Uf_c(\Delta)$. We obtain a second kaleidoscopic coloring $d=f\circ c$, and for all $g\in \Homeo(D_n)$ and $v\in \Br({D}_n)$, 
\[
\sigma_d(g,v)=f\circ c_{g(v)}\circ g\circ c_{v}^{-1}\circ f^{-1}.
\]
We deduce that $\sigma_d(g,v)\in \Gamma$ if and only if $\sigma_c(g,v)\in \Delta$. It now follows that $\Uf_c(\Delta)=\Uf_d(\Gamma)$. Hence, $\Uf(\Delta)\simeq \Uf(\Gamma)$ as Polish groups.

\medskip

For (4) implies (1), suppose (4) holds and observe that $n=m$. Fixing $v\in \Br({D}_n)$, we have the following commutative diagram for all $g\in \Uf_c(\Delta)_{v}$:
\[
    \xymatrix{[n]\ar[r]^{c_v^{-1}}\ar[d]^{\sigma_c(g,v)}& \comp{v} \ar[r]^{\fhi}\ar[d]^{g} & \comp{\fhi(v)} \ar[r]^{d_{\fhi(v)}}\ar[d]^{\fhi g \fhi^{-1}} & [n]\ar[d]^{\sigma_d(\fhi g \fhi^{-1},\fhi(v))} \\
               [n] \ar[r]^{c_v^{-1}}& \comp{v} \ar[r]^{\fhi} & \comp{\fhi(v)} \ar[r]^{d_{\fhi(v)}} & [n] }.
\]
We observe additionally that all maps in the diagram are bijections. Hence,
\[
f=d_{\fhi(v)}\circ \fhi\circ c_v^{-1}:[n]\rightarrow [n]
\]
is a bijection, and moreover,
\[
f\sigma_c(g,v)f^{-1}=\sigma_d(\fhi g \fhi^{-1},\fhi(v))
\]
for all $g\in \Uf_c(\Delta)_{v}$.  

Lemma~\ref{lem:pt_stab_factorization} ensures that $\sigma_c(g,v)$ can take any value in $\Delta$, so $f\Delta f^{-1}\leq \Gamma$. The same argument considering $f^{-1}$ shows conversely that $\Gamma\leq f\Delta f^{-1}$. We conclude that $(\Delta, [n])$ is isomorphic to $(\Gamma,[n])$ as permutation groups.
 
\medskip

We finally argue for (3) implies (4), the most difficult of the implications. Fix $c$ a kaleidoscopic coloring of $D_n$ and $d$ a kaleidoscopic coloring of $D_m$, set $G=\Uf_c(\Delta)$ and $H=\Uf_d(\Gamma)$ and suppose that $\chi:G\rightarrow H$ is an isomorphism of Polish groups. In view of Lemma~\ref{lem:pt_stab_iso} for each $v\in \Br({D}_n)$, there is some $w\in \Br({D}_m)$ such that $\chi(G_{v})=H_{w}$. We thus have a map $\psi:\Br({D}_n)\rightarrow \Br({D}_m)$ such that $\chi(G_{v})=H_{\psi(v)}$, and it follows that this map is a bijection. 

We now argue that $\psi$ respects the betweeness relation. Take $v\neq v'$ in $\Br({D}_n)$ and suppose that $w\in (v,v')$ is a branch point. We may find $v''\in \Br({D}_n)$ such that $w=\kappa(v,v',v'')$. Via Lemma~\ref{lem:fixed_pt_set_tripod}, the fixed point set of $H_{(\psi(v),\psi(v'),\psi(v''))}$ is exactly $\{\psi(v),\psi(v'),\psi(v''),\kappa(\psi(v),\psi(v'),\psi(v''))\}$. On the other hand, $H_{(\psi(v),\psi(v'),\psi(v''))}\leq H_{\psi(w)}$, so 
\[
\psi(w)\in \{\psi(v),\psi(v'),\psi(v''),\kappa(\psi(v),\psi(v'),\psi(v''))\}.
\]
 The only possible value for $\psi(w)$ is $\kappa(\psi(v),\psi(v'),\psi(v''))$. We conclude that $\psi(w)\in (\psi(v),\psi(v'))$, and therefore $\psi$ respects the betweeness relation. 

Applying Proposition~\ref{prop:lifting}, there is a homeomorphism $\fhi:D_n\rightarrow D_m$ such that $\fhi\rest_{\Br({D}_n)}=\psi$, so in particular, $n=m$. Taking $g\in G$ and  $v\in \Br({D}_n)$, 
\[
H_{\fhi(g(v))}=\chi(G_{g(v)})=\chi(gG_{v}g^{-1})=\chi(g)\chi(G_{v})\chi(g)^{-1}=H_{\chi(g)(\fhi(v))}.
\] 
As point fixators fix exactly one point, we conclude that $\fhi(g(v))=\chi(g)(\fhi(v))$. Therefore, $g(v)=\fhi^{-1}\circ\chi(g)\circ \fhi(v)$ for all branch points $v$. As the branch points are dense, we deduce that $\fhi^{-1}\circ \chi(g)\circ \fhi=g$, so $\chi(g)=\fhi g \fhi^{-1}$ for all $g\in G$. That is to say, $\fhi G\fhi^{-1}=H$
\end{proof}

As there is a continuum of non-isomorphic discrete permutation groups, we obtain a large family of non-isomorphic Polish groups.
\begin{cor} 
There is a continuum of non-isomorphic kaleidoscopic groups.
\end{cor}
We also obtain an interesting countable family of non-isomorphic groups.
\begin{cor}
For $n\neq m$ in $\Nbee$, the kaleidoscopic groups $\Uf(1\acts [n])$ and $\Uf(1\acts [m])$ are non-isomorphic.
\end{cor}

\section{Universality}
As noted previously, one inspiration for the present work is the Burger--Mozes universal group for a regular tree $T_n$ with $n\in\Nb_{\geq3}$; see \cite[\S3]{Burger-Mozes1}. An important feature of the Burger--Mozes universal groups is their universality property, {\cite[Proposition 3.2.2]{Burger-Mozes1}}. It turns out that the kaleidoscopic groups enjoy a universality property analogous with the one enjoyed by the Burger--Mozes universal groups.

%
%
%

%

\begin{thm}\label{thm:universal}
For $n\in \Nbee$, if $G\leq \Homeo(D_n)$ is transitive on branch points and has  doubly transitive local action $\Gamma\acts[n]$, then $G\leq \Uf_c(\Gamma)$ for some kaleidoscopic coloring $c$.
\end{thm}

\begin{proof}
Let $(x_k)_{k\in \Nb}$ enumerate $\Br(D_n)$. Fix $d:\comp x_0\rightarrow [n]$ a bijection. We now recursively define a coloring $c$ on every $\comp{y}$ for $y \in X_k$ where $X_k$ is a collection of branch points containing $x_0,\dots,x_k$ such that for each $x\in X_k$ there is $g\in G$ with $g(x)=x_0$ and $c_{x}=d_{x_0}\circ g$. The base case is immediate: we define $c_{ x_0}=d_{x_0}$. 

Suppose we have defined $c$ on $\comp y$ for each $y\in X_k$. Form $Z=X_k\cup \{x_{k+1}\}$. Our coloring $c$ may already be defined on $\comp{x_{k+1}}$. If not, define $c_{x_{k+1}}=c_{x_0}\circ g_{k+1}$ for  some $g_{k+1}\in G$ such that $g_{k+1}(x_{k+1})=x_0$. For each adjacent pair $v,w\in Z$ and $i\neq j $ in $[n]$, choose a distinct $y\in (v,w)\cap \Br(D_n)$. Let $h\in G$ be such that $h(y)=x_0$. Since $G_{y}$ acts doubly transitively on $\comp{y}$, we may find $k\in G_{y}$ such that 
\[
c_{x_0}\circ hk (U_y(v))=i \text{ and } c_{x_0}\circ hk (U_y(w))=j
\]
Put $c_{y}=c_{x_0}\circ hk $. Our coloring $c$ is now defined on $X_{k+1}$ defined to be $Z$ along with all of the elements $y$, and for each $x\in X_{k+1}$, there is $g\in G$ such that $c_{x}=d\circ g$.

Our recursive definition is complete, so we obtain $c$ a coloring of $D_n$. This coloring is moreover kaleidoscopic, and for each $x_i$ there is $g_i\in G$ such that $g_i(x_i)=x_0$ and $c_{x_i}=d\circ g_i$. The elements $g_i$ furthermore have a trivial local action at $x_i$:
\[
\sigma_c(g_i,x_i)=c_{x_0}\circ g_i\circ (c_{x_i})^{-1}=d\circ g_i\circ g_i^{-1}\circ d^{-1}=1.
\]

Take an arbitrary $h\in G$ and $x_i\in \Br(D_n)$ and say $h(x_i)=x_j$. The element $g_jhg_i^{-1}$ is in $G$ and it fixes $x_0$. Therefore, $\sigma_c(g_jhg_i^{-1},x_0)\in \Gamma$ where $\Gamma$ is the local action of $G_{(x_0)}$ on $\comp{x_0}$. On the other hand,
\[
\sigma_c(g_jhg_i^{-1},x_0)=\sigma_c(g_j,x_j)\sigma_c(h,x_0)\sigma_c(g_i^{-1},x_i)=\sigma_c(h,x_i).
\]
We thus deduce that $\sigma_c(h,z)\in \Gamma$ for all $h\in G$ and branch points $z$. Thus, $G\leq \Uf_c(\Gamma)$.
\end{proof}

Our next result shows a doubly transitive group $\Gamma$ ensures \textit{any} coloring $c$ of the dendrite $D_n$ produces the kaleidoscopic group $\Uf(\Gamma)$; see \cite[Proposition 2.8]{CW_17} for the analogous statement for Burger--Mozes universal groups.
\begin{thm}\label{thm:K(G)_ind_color}
Say that $n\in \Nbee$ and $\Gamma\leq \Sym(n)$ is doubly transitive. If $c$ is any coloring of $D_n$, then $\Uf_c(\Gamma)=\Uf(\Gamma)$.
\end{thm}

\begin{proof}
Let $(x_k)_{k\in \Nb}$ enumerate $\Br(D_n)$. We now recursively define a kaleidoscopic coloring $d$ on every $\comp{y}$ for $y \in X_k$ where $X_k$ is a collection of branch points containing $x_0,\dots,x_k$. For the base case, we set $d_{x_0}=c_{x_0}$.

Suppose we have defined $d$ on $\comp y$ for each $y\in X_k$. Form $Z=X_k\cup \{x_{k+1}\}$. Our coloring $d$ may already be defined on $\comp{x_{k+1}}$. If not, define $d_{x_{k+1}}=c_{x_{k+1}}$. For each adjacent pair $v,w\in Z$ and $i\neq j $ in $[n]$, choose a distinct $y\in (v,w)\cap \Br(D_n)$. Since $\Gamma$ acts doubly transitively on $[n]$, we may find $\gamma\in \Gamma$ such that 
\[
\gamma c_{y} (U_y(v))=i \text{ and } \gamma c_{y}(U_y(w))=j.
\]
Put $d_{y}=\gamma c_{y}$. Our coloring $d$ is now defined on $X_{k+1}$ defined to be $Z$ along with all of the elements $y$.

Our recursive definition is complete, so we obtain $d$ a coloring of $D_n$, and this coloring is kaleidoscopic. For each element $g\in \Uf_d(\Gamma)$ and $x\in \Br(D_n)$, we have
\[
\sigma_d(g,x)=d_{g(x)}\circ g\circ (d_{x})^{-1}=\gamma c_{g(x)}\circ g\circ c_x^{-1}\gamma'
\]
where $\gamma$ and $\gamma'$ are some elements of $\Gamma$. It now follows that $\sigma_c(g,x)\in \Gamma$. Hence, $\Uf_d(\Gamma)\leq\Uf_c(\Gamma)$. The converse inclusion is similar.
\end{proof}

 Theorem~\ref{thm:K(G)_ind_color} fails if $\Gamma$ is not doubly transitive. 

\begin{exam}
Fix two distinct end points $\xi^-$ and $\xi^+$ of $D_3$. Let $c$ be a coloring such that for every branch point $z\in [\xi^-,\xi^+]$, $c_z(\xi^-)=1$ and $c_z(\xi^+)=0$. Suppose $c$ is additionally such that for all $z\in \Br(D_n)$, $c_z(\xi^+)=0$. One checks that $\Uf_c(\Zb / 3\Zb)$ does not transitively on pairs of endpoints. On the other hand, Corollary~\ref{cor:U(F)_trans}, shows that any kaleidoscopic group acts transitively on pairs of endpoints. Hence $\Uf_c(\Zb\ / 3\Zb)$ is not a kaleidoscopic group. 
\end{exam}

\section{Cohomology and generosity}\label{sec:coho}
This section exposes how to obtain cohomology for a kaleidoscopic group $\Uf(\Gamma\acts [n])$ out of information of cohomological type about the local permutation group $\Gamma\acts [n]$.

\subsection{Preliminaries on cohomology}
Recall the homogeneous model for general cocycles on an arbitrary set $E$, with values in an abelian group $A$, which will be $A=\Zb$ or $A=\Rb$ in the applications below. Given $q\in\Nb$, a \textbf{$q$-cochain} is any map $\alpha\colon E^{q+1}\to A$. The $q$-cochain $\alpha$ is called \textbf{alternating} if any permutation of its variables only modifies its value by multiplying it with the sign of the permutation. The homogeneous \textbf{coboundary} of $\alpha$ is the $(q+1)$-cochain $d\alpha$ defined by omitting variables as follows:
$$(d\alpha) (x_0, \ldots, x_{q+1}) = \sum_{j=0}^{q+1} (-1)^j\alpha( \ldots, \widehat{x_{j}}, \ldots).$$
Finally, a \textbf{cocycle} is an element of the kernel of $d$. When no further assumption is made, every cocycle is a coboundary. For instance, given any $p\in E$, one checks that a cocycle $\alpha$ satisfies $\alpha=d\beta$ when $\beta$ is defined by
\begin{equation}\label{eq:homotopy}
\beta (x_0, \ldots, x_{q-1}) = \alpha(p, x_0, \ldots, x_{q-1}).
\end{equation}
One of the ways to realize the \textbf{cohomology} $\HH^q(G, A)$ of a group $G$ is as the quotient of the \emph{$G$-invariant} $q$-cocycles on $E=G$ by the coboundaries of $G$-invariant $(q-1)$-cochains. When $A$ is divisible (e.g.\ $A=\Rb$), one can equivalently use alternating cocycles and coboundaries.

The \textbf{bounded} cohomology is obtained by restricting both cocycles and coboundaries to be bounded maps. It therefore comes with a natural \textbf{comparison map} to usual cohomology. Although this map is induced by the inclusion of cochains, it is in general not injective (nor surjective) at the level of cohomology.

\subsection{A local-to-global procedure for cocycles}
The basic construction of this section is as follows. Fix a kaleidoscopic coloring $c$ of $D_n$. Given a $\Gamma$-invariant alternating $2$-cocycle $\Omega\colon [n]^3\to A$, we define a map
$$\omega\colon \Ends(D_n)^3 \longrightarrow A$$
$$
\omega(\xi, \eta, \zeta) =
\begin{cases}
 \Omega(c_a(\xi), c_a(\eta), c_a(\zeta))& \xi,\;\eta\;\text{ and } \zeta\text{ distinct with center }a \\
 0 & \text{ otherwise.}
 \end{cases}
$$

\begin{prop}\label{prop:omega:Omega}
The map $\omega$ is a $\Uf(\Gamma)$-invariant alternating $2$-cocycle.
\end{prop}

\begin{proof}
The issue in checking that $d\omega$ vanishes on $4$-tuples of ends comes from the fact that the definition of $\omega$ depends on the center of a triple, whilst there is in general no single center for four points. As for invariance and alternation, these properties follow by construction, using the corresponding properties of $\Omega$.

Consider thus four ends $\xi, \eta, \zeta, \teta$. We can assume that they are pairwise distinct since otherwise the above issue does not arise. For the same reason, we can assume that the tree spanned in $D_n$ by these ends has two inner nodes of degree three rather than one of degree four (these being the only possibilities for a tree with four leaves). Finally, upon permuting the variables (which we can do since $\omega$ is alternating), we are reduced to the situation where the center $a$ of $\xi, \eta, \zeta$ coincides with the center of $\xi, \eta, \teta$, and the center $b\neq a$ of $\xi, \zeta, \teta$ coincides with the center of $\eta, \zeta, \teta$.
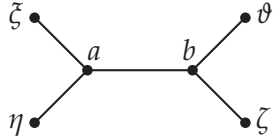
\begin{center}
\setlength{\unitlength}{0.7cm}
\thicklines
\begin{picture}(6,4)
\put(2,2){\line(1,0){2}}
\put(2,2){\line(-1,1){1}}
\put(2,2){\line(-1,-1){1}}
\put(4,2){\line(1,1){1}}
\put(4,2){\line(1,-1){1}}
\put(1,3){\circle*{0.2}}\put(0.5,2.9){$\xi$}
\put(1,1){\circle*{0.2}}\put(0.5,0.9){$\eta$}
\put(5,3){\circle*{0.2}}\put(5.2,2.9){$\teta$}
\put(5,1){\circle*{0.2}}\put(5.2,0.9){$\zeta$}

\put(2,2){\circle*{0.2}}\put(2,2.2){$a$}
\put(4,2){\circle*{0.2}}\put(3.8,2.2){$b$}
\end{picture}
\end{center}
It follows that the colorings at $a$ of $\zeta$ and $\teta$ coincide; therefore, we have
$$\omega(\xi, \eta, \teta) =\omega(\xi, \eta, \zeta).$$
Likewise, the coloring at $b$ shows that
$$\omega( \eta, \zeta, \teta) = \omega(\xi, \zeta, \teta).$$
These two equalities imply indeed the cocycle equation
$$\omega( \eta, \zeta, \teta) - \omega(\xi, \zeta, \teta) +\omega(\xi, \eta, \teta) -\omega(\xi, \eta, \zeta) =0$$
that was to be established.
\end{proof}

Recall that a cocycle that is invariant for a group acting transitively on the set where the cocycle is defined yields canonically a cohomology class for that group, although this cohomology class could be trivial even if the cocycle is not the coboundary of an invariant cochain on that set. Concretely, keeping the notation above, the corresponding invariant cocycle on $\Uf(\Gamma)^3$ can be defined by 
$$(g_0, g_1, g_2) \mapsto \omega(g_0 \xi, g_1 \xi, g_2 \xi),$$
where $\xi$ is an arbitrary element of $\Ends(D_n)$. The transitivity assumption implies that the cohomology class of this cocycle does not depend on the choice of $\xi$. Our source of non-trivial cohomology classes is the following result; we emphasize that the boundedness assumption on $\Omega$ is needed in the proof, although of course it is automatically satisfied for $n<\infty$.

\begin{thm}\label{thm:coho:gen}
Let $A=\Zb$ or $A=\Rb$ and suppose that $\Omega\colon [n]^3\to A$ is a non-zero bounded alternating $\Gamma$-invariant cocycle. Then $\omega$ determines a non-trivial bounded cohomology class in $\HB^2(\Uf(\Gamma), A)$ which remains non-trivial in the usual cohomology $\HH^2(\Uf(\Gamma), A)$.
\end{thm}

\begin{rmk}
We emphasize that $\Omega$ is not assumed to be cohomologically non-trivial as  $\Gamma$-invariant cocycle. Indeed, we shall notably apply Theorem~\ref{thm:coho:gen} to \emph{coboundaries} of $\Gamma$-invariant cochains.
\end{rmk}

\begin{proof}[Proof of Theorem~\ref{thm:coho:gen}]
It suffices to prove the formally stronger statement with $A=\Rb$. Moreover, it suffices to show that the class of $\omega$ in \emph{bounded} cohomology is non-trivial: indeed, the uniform perfectness of $\Uf(\Gamma)$, established in Theorem~\ref{thm:uniform_perfect} above, implies that the natural map
$$\HB^2(\Uf(\Gamma), \Rb) \longrightarrow \HH^2(\Uf(\Gamma), \Rb)$$
is injective, see e.g.\ Corollary~2.11 in~\cite{Matsumoto-Morita}.

In fact, we shall prove the following stronger statement. Let $\Lambda$ be any countable subgroup of $\Uf(\Gamma)$. If the $\Lambda$-action on $D_n$ is dendro-minimal, then the pull-back of our class to $\HB^2(\Lambda, \Rb)$ is non-trivial. To verify that this statement indeed implies the non-vanishing in $\HB^2(\Uf(\Gamma), \Rb)$ it suffices to observe that such a countable group $\Lambda$ exists --- any dense countable subgroup will do since $\Uf(\Gamma)$ acts transitively on branch points (Proposition~\ref{prop:U(1)_trans}) and thus is dendro-minimal.

To prove the statement for $\Lambda$, we recall from~\S8 in~\cite{DM_dendrites} that there exists a non-singular measure $\Lambda$-space $B$ together with a measurable $\Lambda$-equivariant map $\fhi\colon B\to\Ends(D_n)$ such that the $\Lambda$-action on $B$ is amenable in Zimmer's sense and such that the diagonal $\Lambda$-action on $B^2$ is ergodic. We further recall that the cocycle
$$\fhi^*\omega\colon B^3 \longrightarrow  \Rb$$
defined by precomposition with $\fhi$ realizes the class in $\HB^2(\Lambda, \Rb)$ under consideration, see e.g.~(7.5.3) and~(7.2.6) in~\cite{Monod_LNM}. Now the ergodicity of $B^2$ implies that $\fhi^*\omega$ defines a non-vanishing class in $\HB^2(\Lambda, \Rb)$ unless $\fhi^*\omega$ itself vanishes almost everywhere, since any alternating $\Lambda$-invariant measurable map $B^2\to\Rb$ (as needed for coboundaries) must vanish. It thus remains to show that $\fhi^*\omega$ does not vanish almost everywhere. To this end, pick any branch point $a$; the fact that $\Omega$ is not identically zero on $[n]^3$ implies that there are three components $X,Y,Z\in \comp a$ such that $\omega( \xi, \eta, \zeta) \neq 0$ for all end points $\xi\in X$, $\eta\in Y$, $\zeta\in Z$. Consider the measure on $D_n$ which is the image under $\fhi$ of the given measure on $B$; it remains only to observe that this measure is non-zero on $X$, $Y$ and $Z$. This is the case because these components are open, the class of the measure is preserved by $\Lambda$, and $\Lambda$ acts minimally on $D_n$ (by~\cite{DM_dendrites}, Lemma~4.4 and Remark~4.7).
\end{proof}

\subsection{Generosity}
We recall a definition due to P.~Neumann~\cite{Neumann75}.

\begin{defn}
An action of group $G$ on some set is \textbf{generously transitive} if for any $x,y$ in the set there is $g\in G$ such that $g(x)=y$ and $g(y)=x$.
\end{defn}

\begin{rmk}
Generous transitivity implies transitivity and is implied by double transitivity. The action of a cyclic group of order~$\geq 3$ on itself shows that it is not the same as transitivity and the action of the dihedral group of a square on the vertices of that square shows that it is not the same as double transitivity.

Moreover, generous transitivity is unrelated to primitivity: the action of $\Zb/p\Zb$ on itself is primitive but not generous when $p\neq 2$ is prime, and the above-mentioned action of the dihedral group of a square is generous but not primitive.
\end{rmk}

We further recall that two actions of a given group are called \textbf{orthogonal} if the diagonal action on the product is transitive. We shall say that an action is \textbf{semi-generous} if it decomposes into two orthogonal orbits that are both generously transitive.

It turns out that a rather natural property in character theory for finite permutation groups is equivalent to being \emph{either generous or semi-generous}, see Proposition~8.4 in \cite{Neumann75}. Exactly this disjunction appears also in the following result. 

\begin{prop}\label{prop:gen}
The following are equivalent for a permutation group $\Gamma\acts[n]$ with $n\in\Nbee$.
\begin{enumerate}[(i)]
\item The permutation group $\Gamma\acts[n]$ is neither generous nor semi-generous.\label{pt:gen:gen}
\item There exists a non-zero coboundary of a bounded alternating $\Gamma$-invariant map $[n]^2\to \Rb$. \label{pt:gen:R}
\item There exists a non-zero coboundary of an alternating $\Gamma$-invariant map $[n]^2\to \{0, \pm 1\}$. \label{pt:gen:0}
\end{enumerate}
\end{prop}

\begin{proof}
The main point is the implication \eqref{pt:gen:gen}$\Rightarrow$\eqref{pt:gen:0}. Since $\Gamma\acts[n]$ is not generous, we can choose $p, q$ in $[n]$ such that, in $[n]^2$, the orbits $\Gamma (p,q)$ and $\Gamma(q,p)$ are disjoint. Define the map $\Delta\colon [n]^2\to \{0,\pm1\}$ by declaring that $\Delta\equiv 1$ on $\Gamma (p,q)$, that $\Delta\equiv -1$ on $\Gamma (q,p)$ and that $\Delta\equiv 0$ elsewhere. It suffices to show that upon possibly changing our choice of $p,q$ (and changing $\Delta$ accordingly), there exist $x,y,z$ in $[n]$ such that the relation
\begin{equation}
\label{eq:cocycle:n}
\Delta(x,y) = \Delta(x,z) + \Delta(z,y)
\end{equation}
does \emph{not} hold; the coboundary is then $\Omega=d \Delta$.

Suppose thus that~\eqref{eq:cocycle:n} always holds and consider the oriented graph on $[n]$ defined by $\Gamma (p,q)$. The relation~\eqref{eq:cocycle:n} applied to $(x,y)=(p,q)$ shows that every $z\in[n]$ either receives an edge from $p$ or originates an edge to $q$, but not both. This defines a $\Gamma$-invariant partition $Q\sqcup P$ of $[n]$ with $p\in P$ and $q\in Q$. Moreover, the definition of our graph shows that the stabilizer of $p$ acts transitively on $Q$ and the stabilizer of $q$ transitively on $P$. It follows that $P$ and $Q$ are orthogonal $\Gamma$-sets. Since $\Gamma\acts[n]$ is not semi-generous, at least one of these two orbits is not generous. We now modify our choices by taking $p,q$ in that orbit and redefining $\Delta$ accordingly. This time the relation~\eqref{eq:cocycle:n} must be violated because otherwise the same argument would provide a $\Gamma$-invariant partition of the chosen orbit, which is absurd. This establishes \eqref{pt:gen:gen}$\Rightarrow$\eqref{pt:gen:0}.

Since \eqref{pt:gen:0}$\Rightarrow$\eqref{pt:gen:R} is trivial, it remains only to justify \eqref{pt:gen:R}$\Rightarrow$\eqref{pt:gen:gen}. If $\Gamma$ were generous, then there would be no non-zero alternating $\Gamma$-invariant function on $[n]^2$ at all. If on the other hand $\Gamma$ were semi-generous, then there would be a one-dimensional space of such functions, but they would automatically satisfy~\eqref{eq:cocycle:n} everywhere.
\end{proof}

Combining Proposition~\ref{prop:gen} with Theorem~\ref{thm:coho:gen} via Proposition~\ref{prop:omega:Omega}, we conclude:

\begin{cor}\label{cor:gen:nonvan}
Suppose that the permutation group $\Gamma\acts[n]$ is neither generous nor semi-generous. Then $\Uf(\Gamma)$ admits a cocycle with values in $\{0, \pm1, \pm 2\}$ which determines a non-trivial class in $\HH^2(\Uf(\Gamma),\Zb)$ and in $\HB^2(\Uf(\Gamma),\Zb)$. Moreover, these two classes remain non-trivial when viewed as $\Rb$-valued cohomology classes.\qed
\end{cor}

\subsection{More cohomology}\label{sec:more:coh}
The methods introduced above give more than just non-vanishing. Since the correspondence $\Omega\mapsto\omega$ of Proposition~\ref{prop:omega:Omega} is additive, we can reformulate Theorem~\ref{thm:coho:gen} as follows.

\begin{cor}\label{cor:coho:gen}
Let $A=\Zb$ or $A=\Rb$ and denote by $Z^2_{\mathrm{alt},\mathrm{b}}([n], A)^\Gamma$ the group of bounded alternating $\Gamma$-invariant $2$-cocycles on $[n]$. Then the maps
$$Z^2_{\mathrm{alt},\mathrm{b}}([n], A)^\Gamma \longrightarrow \HB^2(\Uf(\Gamma), A) \longrightarrow \HH^2(\Uf(\Gamma), A)$$
are injective.\qed
\end{cor}

As a first application, we can give a quantitative estimate on the size of the second cohomology of kaleidoscopic groups by counting arguments that are particularly simple in the case where $\Gamma$ is trivial. (The triangular number below is to be read as $\infty$ in the case $n=\infty$).

\begin{cor}\label{cor:large:schur}
For $n\in\Nbee$, the image of $\HB^2(\Uf(1\acts[n]), \Zb)$ in the ordinary cohomology $\HH^2(\Uf(1\acts[n]), \Rb)$ spans a space of dimension at least $\frac{(n-1)(n-2)}{2}$.
\end{cor}

This result implies Corollary~\ref{cor:schur:large:intro} from the introduction thanks to the universal coefficient theorem.

\begin{proof}
By Corollary~\ref{cor:coho:gen}, it suffices to compute the rank of $Z^2_{\mathrm{alt},\mathrm{b}}([n], \Zb)$. Suppose $n<\infty$. Since there is no equivariance condition, every cocycle in that group is the coboundary of an alternating map $\beta\colon [n]^2\to \Zb$, as follows e.g.\ from equation~\eqref{eq:homotopy}. The group of such maps $\beta$ has rank $\begin{pmatrix}n\\2\end{pmatrix}=n(n-1)/2$. On the other hand, the kernel of the coboundary map $\beta\mapsto d\beta$ which parametrizes $Z^2_{\mathrm{alt},\mathrm{b}}([n], \Zb)$ consists exactly of the image of all maps $[n]\to \Zb$, again by an application of equation~\eqref{eq:homotopy}. This image is the quotient of $[n]^\Zb$ by the kernel consisting of constant maps, and hence has rank $n-1$. We conclude that $Z^2_{\mathrm{alt},\mathrm{b}}([n], \Zb)$ has rank
$$\frac{n(n-1)}{2} - (n-1) = \frac{(n-1)(n-2)}{2}$$
as claimed. The case $n=\infty$ is a simpler version of the same argument.
\end{proof}

\begin{rmk}
The reader might have noticed that all applications so far have used only cocycles on $[n]$ that are coboundaries. In other words, we used cocycles that are trivial for the cohomology of $\Gamma$ and constructed non-trivial cohomology for $\Uf(\Gamma)$. 

In fact, this is unavoidable when working with $n < \infty$. Indeed, in that case, any cocycle is the coboundary of an invariant $\Rb$-valued cochain; this can be seen by averaging the equation~\eqref{eq:homotopy} over all $p\in [n]$.

Therefore, Proposition~\ref{prop:gen} shows that the condition barring generosity and semi-generosity provides the exact setting where our construction can work for $n < \infty$.
\end{rmk}

Turning to $n=\infty$, there are situations where $\Gamma\acts [n]$ has interesting (non-trivial) bounded cohomology. This allows us to produce cohomology classes even for some examples of kaleidoscopic groups associated to generously transitive local actions.

\begin{exam}
Suppose that $\Gamma$ preserves a cyclic order on $[n]$ and denote by $\Omega\colon[n]^3\to\{0, \pm 1\}$ the corresponding cyclic order cocycle (see e.g.~\cite{Ghys84}). Then one checks that the resulting cocycle $\omega$ for $\Uf(\Gamma)$ is nothing else than the cyclic order cocycle associated to a natural cyclic order on $\Ends(D_n)$ defined by the cyclic order on $[n]$ via the given coloring.

Consider the concrete example of $\Gamma=\mathrm{SL}_n(\Qb)$ acting on the projective line over $\Qb$; thus we identify $[\infty]$ with this projective line. Then Theorem~\ref{thm:coho:gen} produces in particular a non-trivial class in $\HH^2(\Uf(\Gamma), \Rb)$. On the other hand, the local action $\Gamma\acts [\infty]$ is doubly transitive, hence in particular generously transitive.
\end{exam}

This example can be extended to a larger group, giving us access to non-trivial cohomology in arbitrarily high degrees:

\begin{cor}\label{cor:T}
Let $T\acts [\infty]$ be the action of Thompson's circle group $T$ on the dyadic points of the circle.

Then the cyclic order cocycle $\omega$ determines a class in $\HB^2(\Uf(T\acts [\infty]), \Zb)$ such that the images of the cup product $\omega^n$ are non-trivial in
$$\HB^{2n}(\Uf(T\acts [\infty]), \Zb)\kern5mm\text{and in}\kern5mm \HH^{2n}(\Uf(T\acts [\infty]), \Zb)$$
for all $n$.
\end{cor}

\noindent
(We refer to~\cite{Cannon-Floyd-Parry} for a description of the group $T$ and of its action.)

\begin{proof}
We choose a lifting $\iota\colon T\to \Uf(T\acts [\infty])$ as provided by the wreath product structure of Corollary~\ref{cor:wreath}. This is moreover a lifting of permutation groups, so that the corresponding restriction maps
$$\iota^*\colon \HH^{*}(\Uf(T\acts [\infty]), \Zb) \longrightarrow \HH^{*}(T, \Zb)$$
send $[\omega]$ to the class determined by the order cocycle for $T\acts [\infty]$. The latter is known to have non-zero cup powers in $\HH^{2n}(T, \Zb)$ for all $n$ by Théorème~D in~\cite{Ghys-Sergiescu}. Therefore, the statement follows from the naturality of the cup product with respect to the maps $\iota^*$ and to the comparison maps
$$\HB^{*}(\Uf(T\acts [\infty]), \Zb) \longrightarrow \HH^{*}(\Uf(T\acts [\infty]), \Zb).$$
\end{proof}

In fact, the reference~\cite{Ghys-Sergiescu} cited above proves non-vanishing in $\HH^{2n}(T, \Qb)$. Therefore, Corollary~\ref{cor:T} holds also with rational coefficients, which allows us to apply the universal coefficient theorem and deduce:

\begin{cor}\label{cor:T:hom}
The homology  $\HH_{2n}(\Uf(T\acts [\infty]), \Zb)$ is non-trivial for all $n$.\qed
\end{cor}

\section{Acyclicity}
A group $G$ is called \textbf{acyclic} if its homology $\HH_n(G,\Zb)$ vanishes for all $n>0$. This implies the cohomology $\HH^n(G,\Zb)$ and  $\HH^n(G,\Rb)$  vanishes for all $n>0$ thanks to the universal coefficient theorem. The main result of this section is the following.

\begin{thm}\label{thm:acyclic}
The homeomorphism group of the universal Wa\.zewski dendrite $D_\infty$ is acyclic.
\end{thm}

This theorem and its corollary below stand in contrast to the non-vanishing results of Section~\ref{sec:coho}.

\begin{cor}\label{cor:h2:van}
Let $G$ be the homeomorphism group of $D_\infty$. Then the bounded cohomology $\HB^2(G,\Zb)$ and $\HB^2(G,\Rb)$ both vanish.
\end{cor}

\begin{rmk}
Once again, all these (co)homological statements regard $G$ as an abstract group. 
\end{rmk}

\begin{proof}[Proof of Corollary~\ref{cor:h2:van}]
A general fact for any group is that in order to deduce the vanishing of $\HB^2(-,\Rb)$ from the vanishing of $\HH^2(-,\Rb)$, it suffices to know that the group is uniformly perfect; see e.g.\ Corollary~2.11 in~\cite{Matsumoto-Morita}. Therefore, by Theorem~\ref{thm:uniform_perfect}, we conclude that $\HB^2(G,\Rb)$ vanishes. Turning to $\Zb$-valued bounded cohomology, the long exact coefficient sequence (see e.g.\ Proposition~1.1 in~\cite{Gersten92}) shows that it is sufficient to know that $\HH^1(G, \Rb/\Zb)$ vanishes, which is the case since $G$ is perfect.
\end{proof}

The proof of Theorem~\ref{thm:acyclic} uses the tree-like structure of the dendrite to reduce the problem to the stabilizers of finite sets of branch points using techniques from algebraic topology. The simplest case is the stabilizer of a single point, where we will leverage ideas that go back to Mather~\cite{Mather71}, Wagoner~\cite{Wagoner72} and Segal~\cite{Segal78} to prove the following.

\begin{thm}\label{thm:acy:Gx}
The stabilizer of a branch point in the homeomorphism group of $D_\infty$ is acyclic.
\end{thm}

The proof of Theorem~\ref{thm:acy:Gx} will be a variation on the arguments provided by de la Harpe and McDuff in~\cite{dlHMcD} for their proof that $\Sym(\infty)$ is acyclic together with an additional topological ingredient because this stabilizer is a full permutational wreath product of $\Sym(\infty)$ with a group of homeomorphisms.

It will be more involved to treat the stabilizers of other finite sets; the key case is the following.

\begin{thm}\label{thm:acy:ends}
The pointwise stabilizer of two distinct end points in the homeomorphism group of $D_\infty$ is acyclic.
\end{thm}

Turning to the proofs, we begin with the stabilizer of two end points.

\begin{proof}[Proof of Theorem~\ref{thm:acy:ends}]
The proof is in two steps; the first is a variation on~\cite{SanVar} and the second follows faithfully~\cite{dlHMcD}. Therefore we shall use a notation compatible with the case of linear groups considered in~\cite{dlHMcD} and urge the reader to have a copy of~\cite{SanVar} and especially of~\cite{dlHMcD} at hand. We start with the set-up for the two steps.

Fix once and for all distinct end points $\xi_\pm\in D_\infty$. We consider the first-point map $r\colon D_\infty \to [\xi_-, \xi_+]$ and identify, for notational convenience, $[\xi_-, \xi_+]$ with $\Rb\cup\{\pm\infty\}$. A \textbf{dashed line} will refer to a closed subset $D\se \Rb$ of the form $D=\cup_{k\in \Zb} [x_k, y_k]$ subject to the following properties: $\lim_{k\to\pm\infty} x_k = \pm\infty$ and $x_k < y_k < x_{k+1}$ for all $k$. We denote by $\Gr$, standing for ``Grassmannian'', the collection of all pre-images $S=r\inv(D)$ where $D$ is some dashed line. Notice that $S$ is closed in $D_\infty \setminus \{\xi_\pm\}$. Finally, a \textbf{flag} is an infinite nested sequence $F=(S_1 \sep S_2 \sep S_3 \sep \cdots)$ of elements $S_i\in \Gr$ such that each $S_{i+1}$ is contained in the \emph{interior} of $S_i$.

Let $G$ be the pointwise stabilizer of $\{\xi_\pm\}$ in $\Homeo(D_\infty)$. The group $G$ acts on $\Gr$ and on the set of flags. Given a flag $F$ we consider the subgroup $G_\infty$ consisting of all $g\in G$ that fix pointwise $S_i$ for some $i$ depending on $g$.

\smallskip
The first step is to prove that the group $G_\infty$ (which depends of $F$) is acyclic. To this end, we shall apply a modified version of Theorem~1.8 in~\cite{SanVar}, which establishes the acyclicity of a certain type of homeomorphism group of a space $X$. In our case, the group is $G_\infty$ and the space is $X=D_\infty \setminus \{\xi_\pm\}$. The assumptions of Theorem~1.8 in~\cite{SanVar} are phrased in terms of a given \emph{directed} family of open subsets $U$ of $X$; in our case, this family is the increasing sequence of all $U=X\setminus S_i$ as $i\geq 1$ varies. The assumptions postulated in~\cite{SanVar} are of two kinds.

The first is an ``admissibility'' assumption (Definition~1.6 loc.\ cit.), which in our case holds thanks to the patchwork lemma, Lemma~\ref{lem:patchwork2} above. This admissibility is required in~\cite{SanVar} with respect only to sequences of open sets converging to a point, but we emphasize that in our case this convergence is not needed because Lemma~\ref{lem:patchwork2} does not require it as an assumption. Specifically, the admissibility is used in order to apply Lemma~1.4 in~\cite{SanVar}, which in our case is subsumed by Lemma~\ref{lem:patchwork2} above.

The second assumption in~\cite{SanVar} amounts to the following. For any $U=X\setminus S_i$ as above, there should be $g\in G_\infty$ such that all images $g^j(U)$ are disjoint as $j\geq 0$, and such that $g^j(U)$ converges to a point in $X$ as $j\to\infty$. For the reasons discussed above, we can dispense of the condition that $g^j(U)$ converges to a point. Then the existence of $g$ follows by another application of Lemma~\ref{lem:patchwork2}. More precisely, if
$$D_i=\cup_{k\in \Zb} [x_{i,k}, y_{i,k}] \kern2mm \text{and}\kern2mm D_{i+1}=\cup_{k\in \Zb} [x_{i+1,k}, y_{i+1,k}]$$
are the dashed lines corresponding to $S_i$ and $S_{i+1}$, we can patch together a homeomorphism $g$ that fixes $D_{i+1}$ pointwise but satisfies
$$ x_{i,k+1}  < g(y_{i,k}) < x_{i+1,k+1} $$
for all $k$. This implies indeed that the open arc $(g^{j+1}(y_{i,k}), g^{j+1}(x_{i,k+1}))$ lies above the  arc $(g^{j}(y_{i,k}), g^{j}(x_{i,k+1}))$ for all $j\geq 0$ and $k\in \Zb$, taking care of our second modified assumption. Therefore, $G_\infty$ is acyclic.

\smallskip
We turn to the second part of the proof, which consists in showing that $G$ itself is acyclic by analyzing its action on the set of all flags. Here we follow faithfully Section~3 of~\cite{dlHMcD}. That reference was written in the context of linear groups but in such a way that it can be adapted to a number of other settings, as illustrated in Section~4 loc.\ cit. In our case, the adaptation is as follows. For two elements $S, S'\in \Gr$, the notation $S\perp S'$ must be read as $S\cap S' = \varnothing$, which is equivalent to the disjointness of the corresponding dashed lines. The notation $S\oplus S'$ must be read as $S \sqcup S'$. With this interpretation, almost all the arguments from~\cite[\S3]{dlHMcD} can be repeated identically, with one exception. Lemma~8 loc.\ cit.\ is proved using \emph{infinite} direct sums of elements of $\Gr$; this device is not available in our context. Therefore we must provide another proof, and hence state the lemma. The statement reduces immediately (following the reduction of Lemma~7 to Lemma~6 in that reference) to this:

Given flags $F_1, \ldots, F_p$ with $F_m=(S_{m,1} \sep S_{m,2} \sep S_{m,3} \sep \cdots)$, there exists flags $F'_m=(S'_{m,1} \sep S'_{m,2} \sep S'_{m,3} \sep \cdots)$ such that $S'_{m,i} \se S_{m,i}$ and $S'_{m,i} \perp  S'_{n,i}$ for all $1\leq m\neq n\leq p$ and all $i\geq 1$.

In order to prove this statement, we shall use the condition $\lim_{k\to\pm\infty} x_k = \pm\infty$ that we imposed on dashed lines. Let thus $\cup_{k\in \Zb} [x_{m,k}, y_{m,k}]$ be the dashed line corresponding to $S_{m,1}$. The convergence condition allows us to define $S'_{m,1}$ simply by skipping the arcs indexed by sufficiently many $k$ to ensure that the remaining arcs, renumbered as $[x'_{m,k}, y'_{m,k}]$, succeed to each other cyclically as $m$ varies. Specifically, we can skip indices (depending on $m$) so that we have
$$y'_{m,k} < x'_{m+1,k} \ \forall\,1\leq m<p \kern2mm\text{and}\kern2mm y'_{p,k} < x'_{1,k}$$
for all $k$. It follows that all $S_{m,1}$ are pairwise disjoint; a fortiori all $S_{m,i}$ are pairwise disjoint for any given $i$, proving the statement.
\end{proof}

We now consider the easier case of stabilizers of single points.

\begin{proof}[Proof of Theorem~\ref{thm:acy:Gx}]
This proof is a simpler version of the proof of Theorem~\ref{thm:acy:ends} for two reasons. On the one hand, the first step will be a direct application of Theorem~1.8 in~\cite{SanVar} without modifications. On the other hand, the second step can follow~\cite[\S3]{dlHMcD} more closely than for Theorem~\ref{thm:acy:ends} because the equivalent of infinite direct sums in $\Gr$ will be available.

Therefore, since we went into all necessary details in the proof of Theorem~\ref{thm:acy:ends}, we can this time indicate only the changes in the set-up. Let thus $G$ be the stabilizer in $\Homeo(D_\infty)$ of a given branch point $x$. Consider the collection $\ul\Gr$ of moieties $\ul S$ of $\comp x$, that is, of subsets $\ul S\se \comp x$ that are simultaneously infinite and with infinite complement. Define this time $\Gr$ to be the collection of subsets $S\se D_\infty$ given by $S= \cup \ul S$ with $\ul S\in\ul\Gr$. A flag shall refer to a nested sequence $F=(S_1 \supseteq S_2 \supseteq \cdots)$ of elements of $\mathrm{Gr}$ such that $\ul S_{i-1} \setminus \ul S_i$ is infinite for all $i\geq 2$ and such that $\bigcap S_i = \varnothing$. We choose some flag $F$ and denote by $G_\infty$ the group of elements fixing pointwise some $S_i$.

Now the assumptions of Theorem~1.8 in~\cite{SanVar} hold unchanged for the group $G_\infty$ acting on $X=D_\infty\setminus\{x\}$. Likewise, all arguments from Section~3 in~\cite{dlHMcD} can be adapted, usually simplified, to work in this setting much like it is done for $\Sym(\infty)$ in~\cite[\S4]{dlHMcD}.
\end{proof}

At this point, the general case follows:

\begin{cor}\label{cor:acy:F}
Let $F$ be a non-empty finite set of branch points in $D_\infty$. Then the pointwise stabilizer of $F$ in $\Homeo(D_\infty)$ is acyclic.
\end{cor}

\begin{proof}
We consider the tree $[F]$ as a graph (without discarding possible vertices of degree two). Let $V\geq 1$ be the number of its vertices and $E\geq 0$ the number of its edges. We claim that the pointwise stabilizer of $F$ can be decomposed as the direct product of $V$ copies of the stabilizer of a branch-point and of $E$ copies of the stabilizer of two distinct end points. This then implies the statement of the corollary, because the K\"unneth theorem reduces it to a combination of Theorems~\ref{thm:acy:ends} and~\ref{thm:acy:Gx}.

To prove the claim, observe that the various restrictions to the components of $D_\infty$ determined by $[F]$ yields an injective homomorphism from the pointwise stabilizer of $F$ to the product of $V$ copies of the stabilizer of a branch-point and of $E$ copies of the stabilizer of two distinct end points. The fact that this homomorphism is surjective follows from the patchwork statement of Lemma~\ref{lem:patchwork2}.
\end{proof}

\begin{proof}[Proof of Theorem~\ref{thm:acyclic}]
By general homological principles, the acyclicity of $G=\Homeo(D_\infty)$ follows if we find an exact sequence of $\Zb[G]$-modules
\begin{equation}\label{eq:seq:M}
0 \leftarrow \Zb \leftarrow M_0 \leftarrow M_1 \leftarrow M_2 \leftarrow \cdots
\end{equation}
such that (i)~the homology $\HH_p(G, M_q)$ vanishes for all $p\geq 1$ and all $q\geq 0$ and (ii)~the sequence of co-invariants
\begin{equation}\label{eq:seq:H0}
0 \leftarrow \Zb \leftarrow \HH_0(G,M_0) \leftarrow \HH_0(G,M_1) \leftarrow \HH_0(G,M_2) \leftarrow \cdots
\end{equation}
remains exact. Indeed this follows e.g.\ immediately from considering the spectral sequence with first tableau $\HH_p(G, M_q)$.

We implement this strategy using a method introduced in~\cite{BM17_preprint} for bounded cohomology, as follows. Let $L_q\se \Br(D_\infty)^{q+1}$ be the set of $(q+1)$-tuples of branch points that lie on a common arc (which depends of course of the tuple). We consider the $\Zb[G]$-modules $M_q=\Zb[L_q]$ and define boundary maps $\partial_q\colon M_q \to M_{q-1}$ by the familiar formula $\partial_q=\sum_{j=0}^q (-1)^j \partial_{q,j}$ where $\partial_{q,j}$ discards the $j$th variable. The augmentation map $M_0\to\Zb$ is the summation of coefficients. We prove the theorem by establishing that the $\Zb[G]$-modules $M_q$ satisfy all required properties.

\smallskip
We first justify that the sequence~\eqref{eq:seq:M} is exact. To this end, we recall the following completely formal basic fact because we need its explicit proof. Let $X$ be any set and consider the (contractible) full chain complex on $X$ given in degree $q$ by $\Zb[X^{q+1}]$. Let $a,b\in X$ be distinct points and define the map $f\colon X\to X$ by $f(a)=b$ and $f(x)=x$ if $x\neq a$. Then a homotopy between the identity and the corresponding chain map $f_q$ on $\Zb[X^{q+1}]$ is given by the maps
$$h_q= \sum_{r=0}^q (-1)^r h_{q,r}\colon \Zb[X^{q+1}] \longrightarrow  \Zb[X^{q+2}]$$
where $h_{q,r}$ is defined on $x\in X^{q+1}$  by
$$\big(f(x_0), \ldots, f(x_{r-1}), a, b, x_{r+1}, \ldots x_q\big)$$
if $x_r=a$ and $h_{q,r}(x)=0$ if $x_r\neq a$. We now consider a cycle $c\in M_q$ and proceed to show that it is the boundary of an element of $M_{q+1}$. Consider the finite tree spanned in $D_\infty$ by all tuples in the support of $c$; we consider every element of every such tuple as a node of the tree, even if it has degree two. Let now $a$ be a leaf of this tree and $b$ the unique node adjacent to $a$ (if the tree is reduced to $a$ the statement is trivial). Applying the above homotopy with $X=\Br(D_\infty)$, we see that every term that appears is still in the chain complex $M_*$ because of the choice of $a$ and $b$. Indeed, if a tuple of $c$ lies on an arc and contains $a$, then adding $b$ to the tuple still remains on an arc. Therefore, $c$ is bounding modulo $\partial M_{q+1}$ to another cycle $c'$ whose associated tree has strictly less nodes. The statement now follows by induction.

\smallskip
Next we establish property~(i). For $q$ given, there are only finitely many $G$-orbits in $L_q$ described completely by topological configuration of $q+1$ points on an arc, see Proposition~6.1 in~\cite{DM_structure}. Therefore, $M_q$ is isomorphic as a $\Zb[G]$-module to a finite sum of modules of the type $\Zb[G/H]$ for subgroups $H<G$. More precisely, $H$ is the pointwise stabilizer of $q+1$ (not necessarily distinct) branch points. By Eckmann--Shapiro induction, we have
$$\HH_p\big(G, \Zb[G/H]\big)\ \cong\ \HH_p(H, \Zb).$$
The latter vanishes for all $p>0$ by Corollary~\ref{cor:acy:F}, as required.

\smallskip
Finally we turn to property~(ii) and examine the sequence~\eqref{eq:seq:H0}. The co-invariant module $\HH_0(G,M_q)$ is the free $\Zb$-module  $\Zb[L_q/G]$ on the set of $G$-orbits in $L_q$. We describe more precisely the orbit set $L_q/G$ using Proposition~6.1 in~\cite{DM_structure}, namely: it can be identified to the finite set of all configurations (i.e. marked homeomorphism classes) of $q+1$ (numbered) points spanning a (possibly degenerate) arc. We keep in mind that no orientation is prescribed since $G$ can reverse arcs. By contrast, the set $C_q$ of configurations on an \emph{oriented} arc forms just an infinite simplex, so that the sequence
\begin{equation}\label{eq:simplex}
0 \leftarrow \Zb \leftarrow  \Zb[C_0]  \leftarrow \Zb[C_1]  \leftarrow \Zb[C_2]  \leftarrow \cdots
\end{equation}
is exact. Moreover, since $L_q/G$ is the quotient of $C_q$ by an action of $\Zb/2$, we can identify $\HH_0(G,M_q)$ with the corresponding $\HH_0(\Zb/2, \Zb[C_q])$. Therefore, we shall establish property~(ii) by showing that the sequence
\begin{equation}\label{eq:Ep0}
\cdots \leftarrow \HH_0(\Zb/2, \Zb[C_q])\leftarrow\HH_0(\Zb/2, \Zb[C_{q+1}]) \leftarrow \cdots
\end{equation}
is exact at all $q>0$. To this end, consider the spectral sequence whose first tableau is $E_{p,q}^1= \HH_p(\Zb/2, \Zb[C_q])$. Since~\eqref{eq:simplex} is exact, this spectral sequence abuts to $\HH_*(\Zb/2, \Zb)$. Fix some $p>0$. For any $q$, the unique fixed point in $C_q$ (the configuration where all points coincide) gives an inclusion $\Zb\to \Zb[C_q]$ of $\Zb[\Zb/2]$-modules. This inclusion induces an isomorphism
$$\HH_p(\Zb/2, \Zb) \xrightarrow{\ \cong\ } \HH_p(\Zb/2, \Zb[C_q])$$
for all $q\geq 0$ because the complement of $\Zb$ in $\Zb[C_q]$ is a free $\Zb[\Zb/2]$-module (recalling $p\neq 0$). Moreover, these isomorphisms intertwine the differential $E_{p,q}^1\to E_{p,q-1}^1$ to a map
$$\HH_p(\Zb/2, \Zb) \longrightarrow \HH_p(\Zb/2, \Zb)$$
which is the zero map when $q$ is odd and the identity when $q\geq 2$ is even; indeed, the fixed point in $C_q$ has $q+1$ identical coordinates. It follows that $E_{p,q}^2$ vanishes for all $p,q\geq 1$ and that $E_{p,0}^2$ is isomorphic to $\HH_p(\Zb/2, \Zb)$ for all $p$. Since this is the abutment of $E_{p,q}$, it follows $E_{0,q}^2=0$ for all $q>0$. This concludes the proof because $E_{0,q}^2$ is precisely the homology of the sequence~\eqref{eq:Ep0}.
\end{proof}


\bibliographystyle{amsplain}
\bibliography{biblio_DMW}

\end{document}